\documentclass{article}

\usepackage[margin=1in]{geometry}
\usepackage[utf8]{inputenc}
\usepackage[english]{babel}
\usepackage{hyperref}
\usepackage{tikz}
\usepackage{amsmath,amsthm,amssymb}
\usepackage{cleveref,xcolor}
\usepackage{csquotes}

\usepackage[maxbibnames=99]{biblatex}
\usetikzlibrary{decorations.markings,arrows.meta}
\newtheorem{thm}{Theorem}
\newtheorem{cor}[thm]{Corollary}
\newtheorem{prop}[thm]{Proposition}
\newtheorem{lem}[thm]{Lemma}
\newtheorem{conj}[thm]{Conjecture}

\newcommand{\vep}{\varepsilon}
\newcommand{\FF}{\mathbb{F}}
\newcommand{\F}{\mathcal{F}}
\newcommand{\PPr}{P^{(r)}}

\newcommand{\Cr}[1]{C^{(r)}_{#1}}
\newcommand{\CC}{\mathcal{C}}

\newcommand{\kgmcomment}[1]{}
\newcommand{\calebcomment}[1]{}

\newcommand{\st}{\colon\,}
\newcommand{\ceil}[1]{\left\lceil #1 \right\rceil}
\newcommand{\ex}{\mathrm{ex}}

\newcommand{\PSG}{\mathrm{PSG}}
\newcommand{\acyc}{a}
\newcommand{\trans}[2]{\tau_{#1}(#2)}
\newcommand{\thresh}[2]{a_{#1}({#2})}
\newcommand{\canptn}{\rho}
\renewcommand{\emptyset}{\varnothing}
\newcommand{\E}{\mathbb{E}}
\newcommand{\cycle}[1]{\left\langle#1\right\rangle}

\title{Tight paths in fully directed hypergraphs}
\author{Richard C. Devine \and Kevin G. Milans}

\begin{document}

\maketitle
\abstract{
It is well-known that every tournament has a spanning path.  We consider hypergraph analogues.  In an \emph{$r$-uniform fully directed hypergraph}, or \emph{$r$-digraph}, every edge is a list of $r$ distinct vertices.  An $(r,k)$-tournament is an $r$-digraph $G$ such that for every $r$-set $S$ of vertices in $G$, exactly $k$ of the orderings of $S$ are edges in $G$.  A \emph{directed tight path} is an $r$-digraph $G$ whose vertices can be ordered so that the intervals of size $r$ are the edges in $G$.  Let $f(n,r,k)$ be the maximum $s$ such that every $n$-vertex $(r,k)$-tournament contains a tight path on $s$ vertices.  Since every tournament has a spanning path, we have $f(n,2,1)=n$.  

In this paper, we show that the minimum $k$ such that $f(n,r,k)$ tends to infinity with $n$ is in the interval $\left[\left(1-\frac{1}{r}-O(\frac{\log r}{r^2\log\log r})\right)r!, ~\left(1-\frac{1}{r} - \frac{\varphi(r)-1}{r!}\right)r!\right]$ where $\varphi(r)$ is the Euler Totient Function, and we find the exact value when $r\le 5$.  We also show that $\Omega(\sqrt{\log n/\log \log n}) \le f(n,3,3) \le O(\log n)$ and $f(n,3,4)\ge \Omega(n^{1/5})$.
}

\section{Introduction}

In a \emph{directed graph} or \emph{digraph}, each edge is an ordered pair of vertices.  There are several natural generalizations to hypergraphs, such as partitioning each edge into a set of head vertices and a set of tail vertices.  In this paper, an edge in a \emph{fully directed hypergraph} is an (ordered) list of distinct vertices.  To our knowledge, only a few results have been obtained in this model~\cite{BS1984, cameron2016extremal}. 

An \emph{$r$-graph} is an $r$-uniform hypergraph, and an \emph{$r$-digraph} is an $r$-uniform fully directed hypergraph.  Many concepts from digraphs extend naturally to the fully directed hypergraph setting.  If $F$ and $G$ are $r$-digraphs, we say that $F$ is a \emph{subgraph} of $G$, denoted $F\subseteq G$ if there is an injection $h\st V(F)\to V(G)$ such that $(u_1,\ldots,u_r)\in E(F)$ implies $(h(u_1),\ldots,h(u_r)) \in E(G)$.

For each $r$-digraph $F$, we define the \emph{Turán number}, denoted $\ex(n,F)$, to be the maximum number of edges in an $n$-vertex $r$-digraph $G$ that does not contain $F$.  Similarly, if $\F$ is a family of $r$-digraphs, then we define $\ex(n,\F)$ to be the maximum number of edges in an $n$-vertex $r$-digraph $G$ such that $F\not\subseteq G$ for each $F\in\F$.  Brown and Simonovits~\cite{BS1984} studied a generalization of this model in 1984 which allows edges to have multiplicity greater than $1$.  Among other results, they proved that the Turán number for an infinite family of $r$-digraphs is approximated by a finite subfamily.  That is, if $\F$ is a family of $r$-digraphs with uniformly bounded edge multiplicities, then for each positive $\vep$ there exists a finite $\F_0\subseteq \F$ such that $\ex(n,\F_0) - \vep n^r \le \ex(n,\F)\le \ex(n,\F_0)$ .

For $k\le r!$, a \emph{$k$-orientation} of an $r$-graph $G$ is an $r$-digraph $G'$ such that $V(G')=V(G)$ and for each $e\in E(G)$ exactly $k$ of the orderings of vertices in $e$ are edges in $G'$.  An \emph{$(r,k)$-tournament} is a $k$-orientation of a complete $r$-graph.  In our language, the usual notion of a tournament is a $(2,1)$-tournament.  With $r$ fixed and $k$ increasing from $0$ to $r!$, an $(r,k)$-tournament evolves from the empty $r$-digraph to the complete $r$-digraph.  For a fixed $r$, it is natural to ask for the minimum value of $k$ such that every $(r,k)$-tournament has some structure or property.  

In a list of objects $(a_1,\ldots,a_t)$, an \emph{interval} is a sublist of consecutive entries.  An \emph{$r$-interval} is an interval of size $r$.  A \emph{directed tight path} is an $r$-digraph whose vertices can be ordered so that the $r$-intervals are the edges in $G$.  The $s$-vertex directed tight path is denoted $\PPr_s$ and we write $\PPr_s = v_1\ldots v_s$ when $v_1,\ldots,v_s$ is the ordering of $V(\PPr_s)$ whose $r$-intervals are the edges in $\PPr_s$.  Similarly, a \emph{directed tight cycle} is an $r$-digraph $G$ whose vertices can be arranged cyclically so that the intervals of size $r$ are the edges in $G$.  For $s\ge r$, the $s$-vertex directed tight cycle is denoted $C^{(r)}_s$ and we write $C^{(r)}_s = \cycle{v_1,\ldots,v_s}$ when $v_1,\ldots,v_s$ is a cyclic arrangement of $V(C^{(r)}_s)$ whose $r$-intervals are the edges in $C^{(r)}_s$.  In this paper, a \emph{path} and a \emph{cycle} refer to tight paths and tight cycles unless otherwise stated.

Let $f(n,r,k)$ be the maximum integer $s$ such that every $n$-vertex $(r,k)$-tournament contains a copy of $\PPr_s$ as a subgraph.  Rédei's Theorem~\cite{Redei1934} asserts that every $(2,1)$-tournament has an odd number of spanning paths.  It follows that every $(2,1)$-tournament has at least one spanning path, and so $f(n,2,1)=n$.  Note that for $n\ge r-1$, we have $r-1 = f(n,r,0) \le f(n,r,1) \le \cdots \le f(n,r,r!) = n$.  We are interested in several natural thresholds.  In terms of $r$, how large does $k$ need to be before $f(n,r,k)$ tends to infinity with $n$, is linear in $n$, or equals $n$?  These are the \emph{growing}, \emph{linear}, and \emph{spanning path thresholds}, respectively.

In \Cref{sec:2-Thresholds-general-r}, we obtain bounds on the thresholds for $k$ for various growth behaviors of $f(n,r,k)$.  Our main result is to obtain the exact threshold on $k$ for $f(n,r,k) = \omega(1)$ in terms of a simpler combinatorial problem on the \emph{pattern shift graph}, and we use this to show that this threshold is between $\left(1-\frac{1}{r}-O(\frac{\log r}{r^2\log\log r})\right)r!$ and $(1-\frac{1}{r}-\frac{\varphi(r)-1}{r!})r!$, where $\varphi(r)$ is the Euler Totient Function $\varphi(r)=|\{d\in[r]\st \gcd(d,r)=1\}|$ (see \Cref{cor:growing-threshold-bounds}).  Also, if $k>(1-1/r)r!$, then every $(r,k)$-tournament has paths of linear size (see \Cref{thm:linear-thresh}), and if $k>(1-\frac{1}{4(r-1)})r!$, then every $(r,k)$-tournament has spanning paths (see \Cref{thm:hamthrsh}).  We summarize these results in \Cref{fig:sec2-summary}.

\begin{figure}
    \begin{center}
    \renewcommand{\arraystretch}{1.3} 
    \begin{tabular}{l|l|l}
    Range on $k/r!$ & Bounds on $f(n,r,k)$ & References \\
    \hline
    $0 \le k/r! \le 1-\frac{1}{r}-\frac{1}{r^{2-o(1)}}$ & $f(n,r,k)=\Theta(1)$ &  \Cref{cor:growing-threshold-bounds}\\
    $1-\frac{1}{r}-\frac{1}{r^{2-o(1)}} \le k/r! \le 1-\frac{1}{r}-\frac{\varphi(r)}{r!}$ & No nontrivial bounds & \\
    $1-\frac{1}{r}-\frac{\varphi(r)-1}{r!} \le k/r! \le 1-\frac{1}{r}$ & $\omega(1) \le f(n,r,k) \le n$  & \Cref{cor:growing-threshold-bounds}\\
    $1-\frac{1}{r} + \frac{1}{r!} \le k/r! \le  1-\frac{1}{4(r-1)}$ & $\Omega(n) \le f(n,r,k) \le n$ & \Cref{thm:linear-thresh}\\
    $1-\frac{1}{4(r-1)} < k/r! \le 1$ & $f(n,r,k) = n$ & \Cref{thm:hamthrsh}
    \end{tabular}
    \end{center}
    \caption{Summary of results in \Cref{sec:2-Thresholds-general-r} for $r
    \ge 3$.  In the first column, $o(1)$ represents a function that tends to $0$ as $r\to\infty$.  In the second column, the big-Oh notation treats $r$ and $k$ as constant and $n$ as the variable.}\label{fig:sec2-summary}
\end{figure}

In \Cref{sec:small-uniform}, we consider fixed small $r$.  We have that $f(n,3,k) = O(1)$ when $k\le 2$ and $f(n,3,k)=n$ when $k\ge 5$.  Finding sharp bounds on $f(n,3,k)$ is interesting when $k\in\{3,4\}$.  We show that $(\log n)^{1/2-o(1)} \le f(n,3,3) \le O(\log n)$ and that $f(n,3,4) \ge \Omega(n^{1/5})$.  We summarize these results in \Cref{fig:sec3-summary}.

\begin{figure}
    \begin{center}
    \renewcommand{\arraystretch}{1.3} 
    \begin{tabular}{l|l|l}
    $k$ & Bounds on $f(n,3,k)$ & References \\
    \hline
    $0\le k\le 2$ & $f(n,3,k)\le 3$ & \Cref{prop:trivthrsh} \\
    $k=3$ & $(\ln n)^{1/2 - o(1)} \le f(n,3,3) \le O(\ln n)$ & \Cref{thm:(33)-upper} and \Cref{thm:(33)-lower}\\
    $k=4$ & $\Omega(n^{1/5}) \le f(n,3,4) \le n$  & \Cref{thm:34-lower}\\
    $5 \le k \le 6$ & $f(n,3,k) = n$ &  \Cref{prop:35-spanning}
    \end{tabular}
    \end{center}
    \caption{Summary of results in \Cref{sec:small-uniform} for $r = 3$.  In the second column, the big-Oh notation treats $r$ and $k$ as constant and $n$ as the variable.}\label{fig:sec3-summary}
\end{figure}

In \Cref{sec:density}, we consider density thresholds for the emergence of paths in $r$-digraphs.  The \emph{density} of an $n$-vertex $r$-digraph $G$ equals $|E(G)|/n_{(r)}$, where $n_{(r)}$ is the \emph{falling factorial} given by $n_{(r)} = n(n-1)\cdots(n-r+1)$.  The \emph{density threshold} for growing paths $d^*$ is the infimum (in fact, minimum when $r\ge 2$) of the set of $d\in[0,1]$ such that every sufficiently large $n$-vertex $r$-digraph with density at least $d$ has a copy of $P_s^{(r)}$, where $s$ grows with $n$.  In \Cref{lem:density-lower}, we show that $d^*\ge 1-\frac{1}{r}$, and in \Cref{cor:density-upper}, we show that $d^*\le 1-\frac{1}{r}$.

The growing path threshold for $(r,k)$-tournaments is one more than the maximum size of an acyclic set of vertices in a digraph that we call the \emph{pattern-shift graph} (see \Cref{thm:const-path-thresh}).  The \emph{$t$-prefix} of a list is the sublist of its first $t$ elements.  Similarly, the \emph{$t$-suffix} is the sublist of its last $t$ elements.  Let $a = (a_1,\ldots,a_s)$ and $b = (b_1,\ldots,b_s)$, and suppose the entries within each list are distinct.  We say that $a$ and $b$ are \emph{order isomorphic} or \emph{pattern-match} if, for $i<j$, we have $a_i < a_j$ if and only if $b_i < b_j$.  The \emph{pattern-shift graph} of order $r$, denoted $\PSG_r$, is the digraph whose vertices are the permutations of $[r]$ with an edge from $a$ to $b$ if and only if the $(r-1)$-suffix of $a$ and the $(r-1)$-prefix of $b$ pattern-match.  When $r$ is small, we often omit tuple notation for the vertices in the pattern-shift graph, writing $123\in V(\PSG_3)$ instead of $(1,2,3)\in V(\PSG_3)$.  If the $(r-1)$-suffix of $a$ equals the $(r-1)$-prefix of $b$, then we say that $ab$ is a \emph{shift edge}.  The subgraph of $\PSG_r$ of shift edges is a disjoint union of $(r-1)!$ cycles of length $r$; we call these \emph{shift cycles}.  See~\Cref{fig:PSG} for $\PSG_3$.  The pattern-shift graph is an analogue of the well-known de Bruijn graph~\cite{deBruijn1946} (see \cite{West2001} for a textbook presentation).  Close variants of the pattern-shift graph were previously investigated by Asplund and Fox~\cite{ASPLUND2018427} and by Chung, Diaconis, and Graham~\cite{CHUNG199243}.

In this paper, we use $\ln$ for the natural logarithm and $\lg$ for the logarithm in base $2$.

\section{Thresholds for Tournaments}\label{sec:2-Thresholds-general-r}

\subsection{Constant Path Thresholds}
    
Let $G$ be an $(r,k)$-tournament.  When $k=0$, the $r$-digraph $G$ has no edges and hence has paths of size at most $r-1$ and length $0$.  When $k>0$, edges are present and so $G$ has paths of size at least $r$ and length at least $1$.  Our next proposition shows that longer paths are not forced even when $k$ is as large as $r!/3$.

\begin{prop}\label{prop:trivthrsh}
If $0<k\le\frac{1}{3}r!$, then $f(n,r,k)=r$.
\end{prop}
\begin{proof}
Taking a single edge in an $(r,k)$-tournament shows that $f(n,r,k)\ge r$.  For the upper bound, note that $1\le k\le \frac{1}{3}r!$ implies that $r\ge 3$.  Since an $(r,k)$-tournament contains an $(r,k')$-tournament when $k'<k$, it suffices to show the upper bound in the case that $k=r!/3$.  Let $k=r!/3$.  We construct an $(r,k)$-tournament avoiding $\PPr_{r+1}$.  We use $[n]$ for our vertex set and we let $(u_1,\ldots, u_r)$ be an edge if and only if $u_2=\max\{u_1,u_2,u_3\}$.  Note that $v_1\ldots v_{r+1}$ does not form a copy of $\PPr_{r+1}$ since the first $r$-tuple requires $v_2>v_3$ and the second $r$-tuple requires $v_2<v_3$.
\end{proof}

Let $t$ be a small, fixed non-negative integer and let $r$ be a positive integer.  The \emph{threshold for paths of length $t$} is the minimum $k$ such that for sufficiently large $n$, we have that $f(n,r,k)\ge r-1+t$.  Note that when $k$ reaches the threshold value for paths of length $t$, every sufficiently large $(r,k)$-tournament contains a path of length $t$, but when $k$ is less than this threshold value there are arbitrarily large $(r,k)$-tournaments that avoid paths of length $t$.  The threshold for paths of length $0$ occurs at $k=0$, and the threshold for paths of length $1$ occurs at $k=1$.  \Cref{prop:trivthrsh} shows that the threshold for paths of length $2$ is more than $r!/3$.

Our next aim is to find the threshold for paths of length $t$ exactly in terms of an optimization problem in a particular digraph.  

\newcommand{\drawPSG}%
{
	\begin{tikzpicture}[]
			
		\begin{scope}[every node/.style={circle,draw,black,inner sep=1pt,minimum size=5ex},line width=0.7pt,decoration={markings,mark=at position 0.5 with {\arrow{Stealth}}}]
            \foreach \n/\lab in {0/321,1/132,2/312,3/123,4/231,5/213}
            {{
                \pgfmathsetmacro{\ang}{\n*60}
                \node at (\ang:2.5cm) (U\lab) {$\lab$} ;
            }}

            \draw[loop left,-Stealth]  (U123) to (U123) ;
            \draw[loop right,-Stealth] (U321) to (U321) ;

            \begin{scope}[every path/.style={bend left=15}]
            \foreach \u/\v in {%
                    U231/U213, U213/U132, U132/U312, U312/U231,
                    U231/U312, U312/U132, U132/U213, U213/U231%
                }
            {{
                \draw[postaction={decorate}] (\u) to (\v) ;
            }}
            \end{scope}

            \foreach \u/\v in {%
                U312/U123, U123/U132, U213/U123, U123/U231,
                U132/U321, U321/U312, U231/U321, U321/U213%
            }
            {{
                \draw[postaction={decorate}] (\u) to (\v) ;  
            }}
        \end{scope}
	\end{tikzpicture}
}
\begin{figure}
    \begin{center}
    \drawPSG
    \end{center}
    \caption{$\PSG_3$}\label{fig:PSG}
\end{figure}

Note that the $\PSG_r$ is regular with in-degree and out-degree both $r$, and the loops in $\PSG_r$ are at the identity permutation and its reverse.  Note that for each list $(a_1,\ldots,a_r)$ with distinct integer entries, there is exactly one permutation of $[r]$ that matches patterns with $(a_1,\ldots,a_r)$.  We call this permutation of $[r]$ the \emph{canonical pattern} of $(a_1,\ldots,a_r)$ and denote it by $\canptn(a_1,\ldots,a_r)$.

Let $t\ge 1$ and let $G$ be a digraph.  Let $\thresh{t}{G}$ be the maximum size of a set of vertices in $G$ inducing a subgraph with no walk of size $t$.  Note that when $|V(G)|=n$, we have $0 = \thresh{1}{G} \le \thresh{2}{G} \le \cdots \le \thresh{n+1}{G} = \acyc(G)$, where $\acyc(G)$ is the maximum size of a set of vertices in $G$ inducing an acyclic subgraph.  Note that $\thresh{2}{G} = \alpha(G)$, where $\alpha(G)$ is the maximum size of a set of vertices containing no edges or loops.  Note that when $G=\PSG_3$, we have $\thresh{1}{G} = 0$ and $\thresh{t}{G} = 2$ for $t\ge 2$ (achieved by, for example, $\{132, 231\}$).

\begin{thm}\label{thm:const-path-thresh}
    Let $r\ge 1$ and $t\ge 1$.  The threshold for paths of length $t$ is $1+\thresh{t}{\PSG_r}$.  In other words, for $n$ large enough, we have $f(n,r,k) < r-1+t$ if $k=\thresh{t}{\PSG_r}$ and $f(n,r,k) \ge r-1+t$ if $k = 1 + \thresh{t}{\PSG_r}$.
\end{thm}

\begin{proof}
    First, we show that $f(n,r,k) < r-1+t$ when $k=\thresh{t}{\PSG_r}$.  Let $S$ be a set of $\thresh{t}{\PSG_r}$ vertices in $\PSG_r$ that induces a subgraph containing no walk of size $t$.  For large $n$, we construct an $(r,k)$-tournament $G$ on $[n]$ by putting $(u_1,\ldots,u_r)\in E(G)$ if and only if $\canptn(u_1,\ldots,u_r) \in S$.  If $v_1\ldots v_m$ is a walk in $G$, then each $r$-interval $(v_{j+1},\ldots,v_{j+r})$ satisfies $\canptn(v_{j+1},\ldots,v_{j+r}) \in S$.  Moreover, it is clear that the last $r-1$ entries of $\canptn(v_j,\ldots,v_{j+r-1})$ and the first $r-1$ entries of $\canptn(v_{j+1},\ldots,v_{j+r})$ both pattern-match $(v_{j+1},\ldots,v_{j+r-1})$.  
    
    Let $w_j = \canptn(v_j,\ldots,v_{j+r-1})$ and let $m'=m-r+1$.  Note that $w_1\ldots w_{m'}$ is a walk in the subgraph of $\PSG_r$ induced by $S$ of size $m'$.  It follows that $m' < t$ and so $m < r-1+t$, implying the desired upper bound on $f(n,r,k)$.

    Next, we show that for sufficiently large $n$, we have $f(n,r,k) \ge r-1+t$ when $k=1+\thresh{t}{\PSG_r}$.  Let $G$ be an $(r,k)$-tournament on $[n]$.  For each $r$-set $A$, let $f(A) = \{\canptn(e)\st \mbox{$e\in E(G)$ and $e$ is an ordering of $A$}\}$ and note that since $G$ is an $(r,k)$-tournament, for each $r$-set $A$, we have that $f(A)$ is a set of $k$ vertices in $\PSG_r$.  Hence $f\st \binom{[n]}{r} \to \binom{V(\PSG_r)}{k}$ is a $\binom{r!}{k}$-edge-coloring of the complete $r$-graph on $[n]$.  Let $m=r-1+t$.  By Ramsey's Theorem~\cite{ramsey1930problem}, since $n$ is sufficiently large, we obtain an $m$-vertex subgraph $G_0$ which is monochromatic in some color $S\in \binom{V(\PSG_r)}{k}$.  Since $|S| = k > \thresh{t}{\PSG_r}$, it follows that there is a walk $w_1\ldots w_t$ in the subgraph of $\PSG_r$ induced by $S$.  By induction on $t$, we obtain a list of distinct real numbers $x_1,\ldots,x_m$ with $m=r-1+t$ such that $\canptn(x_j,\ldots,x_{j+r-1}) = w_j$ for $1\le j\le t$.  For $t=1$, we may simply take $(x_1,\ldots,x_r)=w_1$.  For $t\ge 2$, we obtain $x_1,\ldots,x_{m-1}$ inductively and then select $x_m$ distinct from $x_1,\ldots,x_{m-1}$ so that $\canptn(x_t,\ldots,x_m) = w_t$.  This is possible since $\canptn(x_{t-1},\ldots,x_{m-1}) = w_{t-1}$ and $w_{t-1}w_t\in E(\PSG_r)$, implying that $(x_t,\ldots,x_{m-1})$ pattern-matches the first $r-1$ entries of $w_t$.  

    We claim that for the ordering $v_1\ldots v_m$ of $V(G_0)$ which pattern-matches $x_1,\ldots,x_m$, we have that $v_1\ldots v_m$ is a path.  Indeed, for $1\le j\le t$ we have that $(v_j, \ldots, v_{j+r-1})$ pattern-matches $(x_j,\ldots,x_{j+r-1})$, and $\canptn{(x_j,\ldots,x_{j+r-1})} = w_j \in S$.  Since $f(A)=S$ for each $r$-set of vertices in $G_0$, including $A=\{v_j,\ldots,v_{j+r-1}\}$, it follows that $(v_j,\ldots,v_{j+r-1})\in E(G_0)$.  Hence $G$ contains a path of size $m$ and so $f(n,r,k)\ge m = r-1+t$ when $n$ is sufficiently large.
\end{proof}

In general, computing $\thresh{t}{\PSG_r}$ seems challenging.  It is convenient to define the complementary \emph{transversal number}, denoted $\trans{t}{G}$, to be the minimum size of a set of vertices in $G$ which intersects every walk in $G$ of size $t$.  Since the complement of a set of vertices avoiding all walks of size $t$ intersects every such walk, always $\thresh{t}{G} + \trans{t}{G} = |V(G)|$.  Similarly, let $\trans{}{G}$ be the minimum size of a set of vertices in $G$ which intersects every cycle in $G$, and note that $\trans{}{G} + \acyc(G) = |V(G)|$.

Our next proposition is simple but very useful.
\begin{prop}\label{prop:max-elt-travel}
    Let $uv$ be an edge in $\PSG_r$ with $u=(u_1,\ldots,u_r)$ and $v=(v_1,\ldots,v_r)$.  Let $i$ be the index such that $u_i = r$ and let $j$ be the index such that $v_j = r$.  Either $j=i-1$, $j=r$, or $i=1$.  Consequently every cycle in $\PSG_r$ contains a vertex $(u_1,\ldots,u_r)$ such that $u_1=r$ or $u_r=r$.
\end{prop}
\begin{proof}
    Suppose that $i>1$ and $j<r$.  Since $r = u_i = \max (u_2,\ldots,u_r)$ and $uv\in E(\PSG_r)$, it follows that $v_{i-1} = \max (v_1,\ldots,v_{r-1})$.  Since $v_r \ne r$, we conclude that $v_{i-1} = r$ and so $j=i-1$.  Since the index of the position containing $r$ cannot decrease indefinitely along edges in a cycle, the rest of the proposition follows.
\end{proof}

Our next aim is to find $\thresh{t}{\PSG_r}$ exactly when $t$ divides $r-1$.  

\begin{prop}\label{prop:trans-ub}
    $\trans{t}{\PSG_r} \le \left(\frac{1}{r} + \ceil{\frac{r-1}{t}}\cdot \frac{1}{r}\right)r!$
\end{prop}
\begin{proof}
    Let $S\subseteq V(\PSG_r)$ be the set of all vertices $(u_0,\ldots,u_{r-1})$ with the property that the index $i$ satisfying $u_i = r$ satisfies either $i=r-1$ or $i\equiv 0\pmod{t}$.  Note that besides the last index $r-1$, there are $\ceil{\frac{r-1}{t}}$ indices $i$ with $0\le i \le r-2$ such that $i\equiv 0\pmod{t}$.  Including the last index, $S$ collects all permutations of $[r]$ which have the element $r$ in one of $1+\ceil{\frac{r-1}{t}}$ positions.  It follows that $|S|=\left(1+\ceil{\frac{r-1}{t}}\right)(r-1)!$.
    
    Suppose that $uv\in E(\PSG_r)$ with $u=(u_0,\ldots,u_{r-1})$ and $v=(v_0,\ldots,v_{r-1})$.  Let $i$ and $j$ be the indices satisfying $u_i=r$ and $v_j=r$.  By \Cref{prop:max-elt-travel}, we have that either $i=0$, $j=r-1$, or $j=i-1$.  Let $u_1\ldots u_k$ be a walk in $\PSG_r$ which is disjoint from $S$, and let $z_i$ index the position of $r$ in $u_i$.  Since $u_1\ldots u_k$ is disjoint from $S$, we have that $0 < z_i < r-1$, and it follows that $z_i$ is strictly decreasing.  Since no $z_i$ is congruent to $0$ modulo $t$, it follows that $k\le t-1$.  Hence $S$ intersects every walk in $\PSG_r$ of size $t$.
\end{proof}

Given a tuple $(a_1,\ldots,a_r)$, the \emph{forward shift operation} produces the tuple $(a_2,\ldots,a_r,a_1)$, and the \emph{backward shift operation} produces $(a_r,a_1,\ldots,a_{r-1})$.  A tuple is a \emph{cyclic shift} of another if it can be obtained by a sequence of zero or more shift operations.  Note that partitioning $V(\PSG_r)$ into the equivalence classes of the cyclic shift relation gives the family $\F$ of shift cycles, where $|\F| = r!/r$ and each $C\in\F$ has size $r$.  Since every cycle transversal includes at least one vertex from each of these cycles, $\trans{}{\PSG_r} \ge |\F| = r!(1/r)$.  

\begin{prop}\label{prop:trans-lb}
    If $t\le r$, then $\trans{t}{\PSG_r} \ge \frac{1}{r}\ceil{\frac{r}{t}} \cdot r!$.
\end{prop}
\begin{proof}
    Let $\F$ be the family of shift cycles in $\PSG_r$.  Let $S$ be a set of vertices that intersects each walk in $\PSG_r$ of size $t$, and let $C\in\F$.  If $k=|S\cap V(C)|$, then we have $r-k\le k(t-1)$ since removing the $k$ vertices in $S\cap V(C)$ from $C$ leaves at most $k$ path components, each with at most $t-1$ vertices.  It follows that each $C\in\F$ contains at least $\ceil{r/t}$ vertices in $S$, and since $\F$ is a disjoint family with $(r-1)!$ cycles, we have $|S| \ge \ceil{\frac{r}{t}} (r-1)!$.
\end{proof}

\begin{thm}\label{thm:PSG-t-const}
    If $t$ divides $r-1$, then $\trans{t}{\PSG_r} = \left(\frac{1}{r} + \frac{1}{t} - \frac{1}{tr}\right)r!$ and $\thresh{t}{\PSG_r} = \left(1 - \frac{1}{r} - \frac{1}{t} + \frac{1}{tr}\right)r!$.
\end{thm}
\begin{proof}
    Suppose that $r-1=tk$ for some integer $k$.  By \Cref{prop:trans-ub}, we have $\trans{t}{\PSG_r} \le \left(\frac{1}{r} + \ceil{\frac{r-1}{t}}\cdot \frac{1}{r}\right)r! = \frac{1+k}{r} r!$.  By \Cref{prop:trans-lb}, we have $\trans{t}{\PSG_r} \ge \frac{1}{r}\ceil{\frac{r}{t}} \cdot r! = \frac{1+k}{r} r!$.  It follows that $\trans{t}{\PSG_r}= \frac{1+k}{r} r! = \left(\frac{1}{r} + \frac{1}{t} - \frac{1}{tr}\right)r!$.  The result on $\thresh{t}{\PSG_r}$ follows from $\trans{t}{\PSG_r} + \thresh{t}{\PSG_r} = |V(\PSG_r)| = r!$.
\end{proof}

\subsection{Growing Path Threshold}

Let $r$ be a positive integer.  The \emph{threshold for growing paths} is the minimum $k$ such that we have that $f(n,r,k) \to \infty$ as $n\to \infty$.  

\begin{thm}\label{thm:growing-path-tourn-threshold}
Let $r$ be a positive integer.  The threshold for growing paths is $1+\acyc(\PSG_r)$.  That is, we have 
\[ f(n,r,k) = 
\begin{cases}
    O(1) & \mbox{if $k\le \acyc(\PSG_r)$} \\
    \omega(1) & \mbox{if $k > \acyc(\PSG_r)$}.
\end{cases}
\] 
\end{thm}
\begin{proof}
    Let $k=\acyc(\PSG_r)$.  We set $t=r!+1$, so that $k=\acyc(\PSG_r) = \thresh{t}{\PSG_r}$.    By \Cref{thm:const-path-thresh}, we have $f(n,r,k) < r - 1 + t$, implying that $f(n,r,k) = O(1)$.  Now, let $k=1+ \acyc(\PSG_r)$.  By \Cref{thm:const-path-thresh}, for each $t$, there exists $n$ such that $f(n,r,k) \ge r-1+t$, implying that $f(n,r,k)\to\infty$ as $n\to\infty$.  
\end{proof}

Since $\acyc(\PSG_r) + \trans{}{\PSG_r} = |V(\PSG_r)| = r!$, obtaining bounds on $\acyc(\PSG_r)$ is equivalent to obtaining bounds on $\trans{}{\PSG_r}$.  We note that the analogous problem for the de Bruijn graph has been solved.  The \emph{$m$-ary de Bruijn graph of order $n$}, denoted $B_{n,m}$, is the directed graph with vertex set $\{0,\ldots,m-1\}^n$ with an edge from $(u_1,\ldots,u_n)$ to $(v_1,\ldots,v_n)$ if and only if $(v_1,\ldots,v_{n-1}) = (u_2,\ldots,u_n)$.  An edge $uv$ in $B_{n,m}$ with $u_1=v_n$ is a \emph{shift edge}, and the spanning subgraph of $B_{n,n}$ formed by the shift edges is a disjoint collection of \emph{de Bruijn-shift cycles}.  Note that the subgraph of $B_{r,r}$ induced by $r$-tuples with distinct entries is the spanning subgraph of $\PSG_r$ formed by the shift edges.  Mykkeltveit~\cite{MYKKELTVEIT197240} proved that $\trans{}{B_{n,2}}$ equals the number of de Bruijn-shift cycles in $B_{n,2}$, and this proof generalizes to $m>2$ (see~\cite{ALVAREZ2024352}).  When $n$ is prime, the number of de Bruijn-shift cycles in $B_{n,m}$ equals $m+(m^n-m)/n$, since all de Bruijn-shift cycles have size $n$ except the $m$ cycles formed by the constant sequences.  Hence $\trans{}{B_{n,m}} = m^n(\frac{1}{n} + \frac{n-1}{nm^{n-1}})$ when $n$ is prime.  

In contrast to the de Bruijn graph, the cycle transversal number $\trans{}{\PSG_r}$ of the pattern-shift graph is larger than $\frac{1}{r}r!$, the number of shift cycles in $\PSG_r$ (see \Cref{thm:cycle-pack}).  Our next theorem shows that $\trans{}{\PSG_r}$ is not much larger.

\newcommand{\clust}[2]{\lambda_{#1}(#2)}
\begin{thm}\label{thm:cycle-trans}
We have $\trans{}{\PSG_r}\le r!\left(\frac{1}{r} + O\left(\frac{\log r}{r^2\log\log r}\right)\right)$.
\end{thm}
\begin{proof}
Let $t$ be a positive integer such that $r\ge \max\{3t+2,3\}$.  We show that $\trans{}{\PSG_r} \le r! \left(\frac{1}{r} + \frac{1}{rt!} + \frac{6(3t+2)}{r(r-1)}\right)$.   We construct a cycle transversal as follows.  Let $A = \{1,\ldots,t+1\} \cup \{r-2t,\ldots,r\}$, so that $A$ contains the indices of the first $t+1$ positions and the last $2t+1$ positions in permutations of $[r]$.  Let $S_0$ be the set of permutations $(u_1,\ldots,u_r)$ of $[r]$ in which $(u_i,u_j) = (r-1,r)$ for some $i,j$ with $\{i,j\}\subseteq A$ and $\{i,j\} \cap \{1,r-t,r\} \ne \emptyset$.  We use the values $r$ and $r-1$ as bookmarks for a pattern, and we show that this pattern persists along cycles in $\PSG_r - S_0$.

Fix a linear ordering $\preceq$ on the permutations of $[t]$.  Given a vertex $u\in V(\PSG_r)$ such that $u=(u_1,\ldots,u_r)$ and $u_j = \ell$ for some $j\le r-t$, the \emph{$\ell$-cluster} of $u$, denoted $\clust{\ell}{u}$, is the canonical pattern $\canptn(u_{j+1}, \ldots,u_{j+t})$.  Let $S_1$ be the set of all $u\in V(\PSG_r)-S_0$ such that $u=(u_1,\ldots,u_r)$, $u_{r-t} = r$, and $\clust{r-1}{u} \succeq \clust{r}{u}$.  Similarly, let $S_2$ be the set of all $u\in V(\PSG_r) - S_0$ such that $u=(u_1,\ldots,u_r)$ with $u_1 = r$ and $\clust{r}{u} \succeq \clust{r-1}{u}$.

Let $S=S_0\cup S_1\cup S_2$.  We claim that $S$ is a cycle transversal.  Let $B$ be the set of $(u_1,\ldots,u_r)\in V(\PSG_r)$ such that $(u_i,u_j)=(r-1,r)$ for some $i,j$ with $\{i,j\}\cap\{1,\ldots,r-t\}\ne\emptyset$.  Note that $B$ is the set of vertices $u$ in $\PSG_r$ such that either $\clust{r}{u}$ or $\clust{r-1}{u}$ is defined.  We define a function $g$ from $B$ to the permutations on $[t]$ as follows.  Given $u\in B$ with $u=(u_1,\ldots,u_r)$, if the index $j$ with $u_j=r$ satisfies $j\le r-t$, then we set $g(u)=\clust{r}{u}$.  Otherwise, we set $g(u)=\clust{r-1}{u}$.

Let $C$ be a cycle in $\PSG_r$.  Next, we show that $C$ contains a vertex in $S_0\cup B$.  By \Cref{prop:max-elt-travel}, there exists $u\in V(C)$ with $u=(u_1,\ldots,u_r)$ such that $u_1 = r$ or $u_r=r$.  If $u_1 = r$, then $u\in B$.  Otherwise, $u_r=r$.  Let $i$ be the index with $u_i=r-1$.  If $i \ge r-2t$, then $u\in S_0$.  Otherwise $u\in B$.  Hence every cycle in $\PSG_r - S_0$ contains a vertex in $B$.

Let $u=(u_1,\ldots,u_r)$ and $v=(v_1,\ldots,v_r)$ with $(u_i,u_j) = (r-1,r)$.  We show that if $uv\in E(\PSG_r)$ with $u,v\not\in S$ and $u\in B$, then $v\in B$ and $g(u)\preceq g(v)$.  Moreover, when $j=1$ or $j=r-t+1$, we have $g(u) \prec g(v)$.

Case 1: $j\le r-t$.  Note that $g(u)=\clust{r}{u}$.  

Subcase 1(a): Suppose that $j > 1$.  Since $u_j$ is the maximum entry in $(u_2,\ldots,u_r)$ and $uv\in E(\PSG_r)$, it follows that $v_{j-1}$ is the maximum entry in $(v_1,\ldots,v_{r-1})$.  If also $v_r < r$, then we have that $v_{j-1} = r$, and so $g(v)=\clust{r}{v}=\clust{r}{u} = g(u)$.  Otherwise, $j > 1$ and $v_r = r$, and since $v_{j-1}$ is the maximum entry in $(v_1,\ldots,v_{r-1})$, we have that $v_{j-1} = r-1$.  We have $g(v) = \clust{r-1}{v} = \clust{r}{u}=g(u)$.  

Subcase 1(b): Suppose $j=1$.  We show that $g(u) \prec g(v)$.  Note that since $j=1$ and $u\not\in S_0$, we have that $i\not\in A$, implying $t+1<i<r-2t$.  Since $u_i$ is the maximum entry in $(u_2,\ldots,u_r)$, it follows that $v_{i-1}$ is the maximum entry in $(v_1,\ldots,v_{r-1})$.  If $v_r < r$, then we have $v_{i-1} = r$ and so $g(v)=\clust{r}{v} = \clust{r-1}{u}$.  Since $u_1 = r$ and $u\not\in S_2$, we have $\clust{r-1}{u} \succ \clust{r}{u} = g(u)$, and so $g(v)\succ g(u)$.  If both $u_1 = r$ and $v_r = r$, then $v$ is obtained from $u$ by a cyclic shift, and we have $g(v) = \clust{r-1}{v} = \clust{r-1}{u} \succ \clust{r}{u} = g(u)$. 

Case 2: $j > r-t$.  Note that since $u\in B$ but $\clust{r}{u}$ is undefined, we have that $i\le r-t$ and $g(u)=\clust{r-1}{u}$.  Also, $v_r \le r-2$, since $v_r\in \{r-1,r\}$ and $v_{j-1} \ge u_j - 1 \ge r-1$ would contradict $v\not\in S_0$.  Since $v_r < r$, it follows that $v_{j-1}=r$.  Since $u\not\in S_0$ and $r$ appears in the last $t$ positions of $u$, it follows that $i \ge 2$.  Also, since $u_i$ is the second largest element in $(u_2,\ldots,u_r)$, it follows that $v_{i-1}$ is the second largest element in $(v_1,\ldots,v_{r-1})$.  Together with $v_r \le r-2$, it follows that $v_{i-1} = r-1$.  So $(v_{i-1},v_{j-1}) = (r-1,r)$ and therefore $\clust{r-1}{u} = \clust{r-1}{v}$.

Subcase 2(a): If $j-1 > r-t$, then since $v_{j-1} = r$ we have that $g(v)=\clust{r-1}{v} = \clust{r-1}{u} = g(u)$.  

Subcase 2(b): Otherwise $j-1=r-t$.  We show that $g(u)\prec g(v)$.  Since the index of the position in $v$ with value $r$ is $j-1$ and $j-1=r-t$, it follows that $g(v)=\clust{r}{v}$.  Also, since $v\not\in S_1$ and $v_{r-t} = r$, we have that $\clust{r-1}{v} \prec \clust{r}{v}$.  It follows that $g(v) = \clust{r}{v} \succ \clust{r-1}{v} = \clust{r-1}{u} = g(u)$.

Suppose for a contradiction that $C$ is a cycle in $\PSG_r - S$.  Since $C$ contains a vertex $u\in B$, it follows that $V(C) \subseteq B$.  Also, since $g(x) \preceq g(y)$ for each $xy\in E(C)$, it follows that $g(x)=g(y)$ for all $x,y\in V(C)$.  Note that if $uv\in E(C)$ with $u=(u_1,\ldots,u_r)$, $v=(v_1,\ldots,v_r)$, and $u_j = r$, then either $j=1$, or $v_{j-1} = r$, or $v_r = r$.  The case $j=1$ is excluded, since then $g(u) \prec g(v)$.  It follows that $C$ has a vertex with $r$ in the last position.  Choose $w\in V(C)$ with $w = (w_1,\ldots,w_r)$ such that the index of the position $j$ with $w_j = r$ is minimized subject to $j\ge r-t+1$.  Let $w'$ be the successor of $w$ in $C$, with $w'=(w'_1,\ldots,w'_r)$.  Since $j=r-t+1$ would imply $g(w)\prec g(w')$, we have $j > r-t+1$.  Note that $w'_r < r$, or else $w'_r = r$ and $w'_{j-1} = r-1$, contradicting $w'\not\in S_0$.  Hence it follows that $w'_{j-1} = r$ and $j-1 > r-t$, contradicting the selection of $w$.  It follows that $S$ is a cycle transversal.

It remains to bound $|S|$.  Note that $|A| = (t+1) + (2t+1) = 3t+2$.  Since each permutation in $S_0$ puts $r$ or $r-1$ in a position indexed by $1, r-t, r$ and the other in a position indexed by $A$, we have that $|S_0| \le 6\cdot |A|\cdot (r-2)! \le 6(3t+2)(r-2)!$.  

Next, we bound $S_1$ and $S_2$.  Let $p$ be the probability that $\pi_1 \preceq \pi_2$ when permutations $\pi_1$ and $\pi_2$ of $[t]$ are chosen uniformly and independently at random.  Conditioning on whether or not $\pi_1$ and $\pi_2$ are distinct gives $p=\frac{1}{2}\cdot (1-\frac{1}{t!}) + 1\cdot \frac{1}{t!} = \frac{1}{2} + \frac{1}{2(t!)}$.  

Choose $u\in V(\PSG_r)$ with $u=(u_1,\ldots,u_r)$ uniformly at random.  We bound $\Pr(u\in S_1)$ by conditioning on the event $i\in A$, where $i$ is the index such that $u_i=r-1$.  Note that $\Pr(u\in S_1~|~i\in A) = 0$, since $u\in S_1$ requires that $u_{r-t} = r$ and that $u\not\in S_0$.  Also, we have $\Pr(u\in S_1~|~i\not\in A) = \frac{1}{r-1}\cdot p$ since then $u_{r-t} = r$ with probability $1/(r-1)$, and, conditioned on $i\not\in A$, the $(r-1)$ and $r$ clusters in $u$ are disjoint and so $\clust{r-1}{u}$ and $\clust{r}{u}$ are independently and uniformly distributed among all permutations of $[t]$.  Hence $\Pr(u\in S_1) = \Pr(u\in S_1 | i\in A)\cdot \Pr(i\in A) + \Pr(u\in S_1 | i\not\in A)\cdot \Pr(i\not\in A) \le 0 + \frac{p}{r-1}\cdot \frac{r-|A|}{r} < \frac{p}{r}$.  A similar computation shows that $\Pr(u\in S_2) < \frac{p}{r}$.  Hence $|S| < r!(\frac{6(3t+2)}{r(r-1)} + \frac{2p}{r}) \le r!(\frac{6(3t+2)}{r(r-1)} + \frac{1}{r}\cdot(1+\frac{1}{t!}))$.  When $t=O(\log r/\log\log r)$, we obtain $|S|\le r!\left(\frac{1}{r} + O\left(\frac{\log r}{r^2\log\log r}\right)\right)$.
\end{proof}

Recall that when applied to $(a_1,\ldots,a_r)$, the forward shift operation produces $(a_2,\ldots,a_r,a_1)$ and the backward operation produces $(a_r,a_1,\ldots,a_{r-1})$.  To obtain lower bounds on $\trans{}{\PSG_r}$, we find disjoint cycles.  We begin with the family $\F$ of shift cycles and obtain a slight improvement by replacing certain cycles in $\F$ with a pair of cycles.

\begin{lem}\label{lem:second-chord}
    Let $u,v\in V(\PSG_r)$.  Let $u'$ be the vertex obtained from $u$ by applying the forward shift operation, and let $v'$ be the vertex obtained from $v$ by applying the backward shift operation.  We have $uv\in E(\PSG_r)$ if and only if $v'u'\in E(\PSG_r)$.
\end{lem}
\begin{proof}
Let $u=(u_1,\ldots,u_r)$ and $v=(v_1,\ldots,v_r)$.  We have $u'=(u_2,\ldots,u_r,u_1)$ and $v'=(v_r,v_1,\ldots,v_{r-1})$.  We have that $uv\in E(\PSG_r)$ if and only if $\canptn(u_2,\ldots,u_r)=\canptn(v_1,\ldots,v_{r-1})$ if and only if $v'u'\in E(\PSG_r)$.
\end{proof}

\begin{lem}\label{lem:split-cycle}
Let $r\ge 2$ and let $C$ be a shift cycle in $\PSG_r$.  If $C$ has a chord or loop, then $V(C)$ can be partitioned into two sets that contain cycles.
\end{lem}
\begin{proof}
Let $C=\cycle{u_0,\ldots,u_{r-1}}$.  If $C$ has a loop $u_iu_i$, then \Cref{lem:second-chord} implies $u_{i-1}u_{i+1}\in E(\PSG_r)$ (subscript arithmetic modulo $r$) and we may partition $V(C)$ into the singleton $\{u_i\}$ and the vertices of the cycle $\cycle{u_{i+1},\ldots,u_{i-1}}$.  Let $u_iu_j$ be a chord of $C$, and note that $j\not\equiv i+1 \pmod{r}$.  By \Cref{lem:second-chord}, we have that $u_{j-1}u_{i+1}\in E(\PSG_r)$.  Following $C$ from $u_j$ to $u_i$ and traversing $u_iu_j$ completes a cycle $C_1$.  Since $j\not\equiv i+1\pmod{r}$, we have that $V(C_1)$ is a proper subset of $V(C)$, and in particular $u_{i+1}\not\in V(C_1)$.  Following $C$ from $u_{i+1}$ to $u_{j-1}$ and traversing $u_{j-1}u_{i+1}$ completes a second cycle $C_2$ that is disjoint from $C_1$.  
\end{proof}

\begin{lem}\label{lem:chorded-shift-cycles-count}
The number of shift cycles in $\PSG_r$ that contain chords or loops equals $\varphi(r)$, where $\varphi(r)$ is the Euler Totient Function given by $\varphi(r)=|\{d\in [r]\st \gcd(d,r) = 1\}|$.
\end{lem}
\begin{proof}
Let $d\in[r]$ be relatively prime to $r$, and let $x\in V(\PSG_r)$ be the vertex such that $x=(x_0,\ldots,x_{r-1})$ with $x_s \equiv ds \pmod{r}$.  Let $u_j$ be the vertex in $\PSG_r$ obtained from $x$ by applying the forward shift operation $j$ times, and let $C$ be the shift cycle $\cycle{u_0,\ldots,u_{r-1}}$.  Let $u_j = (u_{j,0},\ldots,u_{j,r-1})$ and observe that $u_{j,s} = u_{0,s+j} = x_{s+j}$.  Let $d'$ be the inverse of $d$ modulo $r$, and note that $x_0 = r$ and $x_{d'} = 1$.  We claim that $u_{d'}u_{1}\in E(\PSG_r)$.  Indeed, we have $u_{d',s} = x_{d' + s} \equiv d(d' + s) \equiv 1 + ds\pmod{r}$ and $u_{1,s-1} = x_{s} \equiv ds\pmod{r}$.  It follows that $u_{d',0}\equiv 1\pmod{r}$ implying $u_{d',0}=1$, and similarly $u_{1,r-1} \equiv 0\pmod{r}$ implying $u_{1,r-1} = r$.  Also, we have that $u_{d',s} -1 \equiv u_{1,s-1}\pmod{r}$.  Since $2\le u_{d',s} \le r$ and $1\le u_{1,s-1} \le r-1$ for $1\le s\le r-1$, it follows that $u_{d',s}-1 = u_{1,s-1}$ for $1\le s\le r-1$.  Since the $(r-1)$-suffix of $u_{d'}$ matches the pattern of the $(r-1)$-prefix of $u_1$, we have that $u_{d'}u_1\in E(\PSG_r)$ as claimed.  Note that $u_{d'}u_1$ is a chord or a loop, since $1-d'\equiv 1\pmod{r}$ if and only if $d'\equiv 0\pmod{r}$, which is impossible since $d'$ is the inverse of $d$.  Allowing $d$ to range over the integers in $[r]$ that are relatively prime to $r$ produces distinct cycles; indeed, given such a cycle $C$, we have that $d$ is the value that follows $r$ in each vertex of $C$.    

Conversely, let $C$ be a shift cycle in $\PSG_r$ that contains a chord or a loop.  Let $C=\cycle{u_0,\ldots,u_{r-1}}$ with $u_j = (u_{j,0},\ldots,u_{j,r-1})$ for each $j$.  We may assume the vertices of $C$ are indexed so that $u_{0,0} = r$.  Let $x=u_0$ with $x=(x_0,\ldots,x_{r-1})$.  As above, we have $u_{j,s} = x_{j+s}$.  Let $u_iu_j \in E(\PSG_r)$ with $j-i\not\equiv 1\pmod{r}$.  The $(r-1)$-suffix of $u_i$ is $(u_{i,1},\ldots,u_{i,r-1})$ or equivalently $(x_{i+1}, \ldots, x_{i+r-1})$; this is the list $\sigma_1$ obtained from the natural cyclic order on $x$ by deleting $x_i$ and recording the remaining $r-1$ elements in order.  Similarly, the $(r-1)$-prefix of $u_j$ is $(u_{j,0},\ldots,u_{j,r-2})$ or equivalently $(x_j,\ldots,x_{j+r-2})$; this is the list $\sigma_2$ obtained from the natural cyclic order on $x$ by deleting $x_{j-1}$ and recording the remaining elements in order.  Note that $x_i$ and $x_{j-1}$ are distinct entries of $x$.  Since $u_iu_j\in E(\PSG_r)$, we have that $\sigma_1$ and $\sigma_2$ are both lists of size $r-1$ that pattern-match.  Note that $1$ cannot appear in both $\sigma_1$ and $\sigma_2$, since then the minimum element $1$ would appear at distinct indices in $\sigma_1$ and $\sigma_2$.  Similarly, $r$ cannot appear in both $\sigma_1$ and $\sigma_2$.  It follows that $\{x_i,x_{j-1}\} = \{1,r\}$.  We may assume without loss of generality that $x_i = 1$ and $x_{j-1} = r$ (implying $j=1$ as $x_0 = u_{0,0} = r$), since otherwise we may apply the same argument to $u_{j-1}u_{i+1}\in E(\PSG_r)$, which interchanges $i$ and $j-1$.

Note that $\sigma_1 = (x_{i+1},\ldots,x_{i+r-1})$ and $\sigma_1$ is a permutation of $\{2,\ldots,r\}$ since $x_i = 1$.  Similarly, $\sigma_2 = (x_j,\ldots,x_{j+r-2}) = (x_1,\ldots,x_{r-1})$ and $\sigma_2$ is a permutation of $\{1,\ldots,r-1\}$ since $x_0 = r$.  Since $\sigma_1$ and $\sigma_2$ pattern match, we have that $x_{i+t} - 1 = x_{t}$ for $1\le t\le r-1$.  Also, when $t=0$, we have $x_{i+0} - 1 = x_i - 1 = 0$ and $x_0 = r$, implying that $x_{t+i} \equiv x_{t} + 1\pmod{r}$ for all $t$.  Iterating this congruence gives $x_{\ell i} \equiv x_{(\ell -1)i} + 1 \equiv \cdots \equiv x_0 + \ell \equiv r + \ell \equiv \ell\pmod{r}$.  Since $\ell$ ranges over all congruence classes modulo $r$, so must $\ell i$.  It follows that $i$ and $r$ are relatively prime; let $d$ be the inverse of $i$ modulo $r$.  Setting $\ell = sd$, we have that $s$ ranges over all congruence classes with $\ell$ and so $x_{\ell i} \equiv \ell\pmod{r}$ becomes $x_s \equiv sd\pmod{r}$.  It follows that $C$ is one of the $\varphi(r)$ cycles constructed above.
\end{proof}

Combining these results, we obtain the following.

\begin{thm}\label{thm:cycle-pack}
    Let $r\ge 2$.  We have $\trans{}{\PSG_r} \ge r!(\frac{1}{r} + \frac{\varphi(r)}{r!})$.
\end{thm}
\begin{proof}
Let $\F$ be the family of shift cycles in $\PSG_r$ and note that $|\F| = r!/r$.  By \Cref{lem:chorded-shift-cycles-count}, we have that $\varphi(r)$ cycles in $\F$ have chords or loops, and by \Cref{lem:split-cycle}, each of these can be replaced with a pair of cycles, giving a family of $r!/r + \varphi(r)$ disjoint cycles in $\PSG_r$.
\end{proof}

\begin{cor}\label{cor:growing-threshold-bounds}
    Let $r\ge 3$.  We have $1-\frac{1}{r} - O(\frac{\log r}{r^2\log\log r}) \le \acyc(\PSG_r)/r! \le 1 - \frac{1}{r} - \frac{\varphi(r)}{r!}$.  Hence, the threshold for growing paths in $r$-uniform tournaments is in the range $\left[\left(1-\frac{1}{r}-O(\frac{\log r}{r^2\log\log r})\right)r!, ~\left(1-\frac{1}{r} - \frac{\varphi(r)-1}{r!}\right)r!\right]$.
\end{cor}
\begin{proof}
Recall that $\trans{}{\PSG_r} + \acyc(\PSG_r) = |V(\PSG_r)| = r!$.  The bounds on $\acyc(\PSG_r)$ arise by taking $t=\Theta(\log r/\log\log r)$ in \Cref{thm:cycle-trans} and invoking \Cref{thm:cycle-pack}.  By \Cref{thm:growing-path-tourn-threshold}, the threshold for growing paths equals $1+\acyc(\PSG_r)$.
\end{proof}

\subsection{Linear Path Threshold}

In this section, we show that when $k \ge (1-\frac{1}{r}+\frac{1}{r!})r!$, an $(r,k)$-tournament has paths of linear size.  The following proposition is well-known and considered folklore.  We include the short proof for completeness.

\begin{prop}\label{prop:avg-to-min-deg}
Let $G$ be an $r$-graph with average degree $d$.  We have that $G$ has a subgraph with minimum degree at least $d/r$.  Moreover, if $r\ge 2$ and $d>0$, then $G$ contains a subgraph with minimum degree larger than $d/r$.
\end{prop}
\begin{proof}
    If $d=0$, then the claim is trivial.  If $r=1$, then a vertex of maximum degree induces a subgraph with minimum degree at least $d$.  So assume $r\ge 2$ and $d>0$.
    
    Let $G_0$ be the smallest subgraph of $G$ with average degree at least $d$.  Let $n=|V(G_0)|$, and let $m=|E(G_0)|$.  Since $d>0$ and $r\ge 2$, it follows that $G_0$ has an edge and so $n\ge r \ge 2$.  Note that $G_0$ has average degree $rm/n$, with $rm/n \ge d$.  Let $k=\delta(G_0)$ and let $u$ be a vertex in $G_0$ with degree $k$.  Note that $G_0 - u$ has $m-k$ edges and $n-1$ vertices, so that $G_0-u$ has average degree $r(m-k)/(n-1)$.  By extremality of $G_0$, we have $r(m-k)/(n-1) < d \le rm/n$.  It follows that $k > m/n \ge d/r$.
\end{proof}

Our next lemma shows that $r$-digraphs without long paths must contain large subgraphs without closed walks.

\begin{lem}\label{lem:acyclic}
Let $G$ be a $n$-vertex $r$-digraph and suppose that $P^{(r)}_{s+1} \not\subseteq G$.  There is an induced subgraph $H$ such that $|V(H)|\ge (1-o(1))(n/(rs))^{1/(r-1)}$, where the $o(1)$ term depends only on $n$ and $r$ and, for each fixed $r$, approaches $0$ as $n\to \infty$, and every walk in $H$ has size at most $\max\{r-1,s\}$.  In particular, $H$ has no closed walks or cycles.
\end{lem}
\begin{proof}
For each ordered $(r-1)$-tuple of vertices $\sigma$, let $P_\sigma$ be a maximum path in $G$ whose last $r-1$ vertices are as listed in $\sigma$.  Let $s_\sigma = |V(P_\sigma)|$.  Note that for each $\sigma$, we have $r-1\le s_\sigma \le s$.  An ordered $r$-tuple $\pi$ is \emph{good} if $w\in V(P_\sigma)$ where the first $r-1$ entries in $\pi$ are $\sigma$ and the last entry in $\pi$ is $w$.  An unordered $r$-set $S$ is \emph{good} if at least one of its $r!$ permutations is good.

Let $m$ be the number of good (unordered) $r$-sets in $G$.  Since $\sum_{\sigma} |V(P_\sigma)| \ge m$, it follows that $|V(P_{\sigma})| \ge m/n^{r-1}$ for some $(r-1)$-tuple $\sigma$.  Therefore $m \le sn^{r-1}$.  Let $T$ be a maximum set of vertices in $G$ that contains no good $r$-set, and let $t=|T|$.  Since every $(t+1)$-set of vertices contains a good $r$-set, it follows from the de Caen bound~\cite{deCaen} on the Turán number of $K^{(r)}_{t+1}$ that $m\ge (1-o(1)) \frac{\binom{n}{r}}{\binom{t}{r-1}}$.  Therefore $(1-o(1))\frac{\binom{n}{r}}{\binom{t}{r-1}} \le m\le sn^{r-1}$, implying $(1-o(1))\frac{n_{(r)}}{rt_{(r-1)}} \le sn^{r-1}$.  It follows that $(1-o(1))\left(\frac{n^r}{rs}\right)\left(\frac{n_{(r)}}{n^r}\right)\le n^{r-1}t_{(r-1)} \le n^{r-1}t^{r-1}$, and so $t\ge \left[(1-o(1))\left(\frac{n_{(r)}}{n^r}\right)\left(\frac{n}{rs}\right)\right]^{1/(r-1)} = (1-o(1))\left(\frac{n}{rs}\right)^{1/(r-1)}$.

Let $H=G[T]$.  We show that every walk in $H$ has size at most $s$.  Note that if $\pi\in E(H)$, then the underlying set of $\pi$ is not good, and so $w\not\in V(P_\sigma)$ where $\pi$ consists of $\sigma$ followed by $w$.  Therefore $w$ extends a maximum path in $G$ ending in $\sigma$ to a longer one in $G$ ending in $\sigma'$, where $\sigma'$ is the last $r-1$ vertices in $\pi$.  It follows that $s_{\sigma'}\ge s_{\sigma} + 1$.  Let $W$ be a walk in $H$, with $W=u_1\ldots u_\ell$.  If $\ell\le r-1$, then the claim holds.  Otherwise $\ell\ge r$ and if $\sigma$ is the $(r-1)$-interval of $W$ ending at $u_j$, then $s_\sigma \ge j$.  In particular, when $\sigma$ ends at $u_\ell$, we have $\ell\le s_\sigma \le s$.
\end{proof}

For $n\ge r$, the \emph{cycle} $\Cr{n}$ is the $n$-vertex $r$-digraph whose vertices are arranged cyclically and whose edges are the intervals of size $r$.  Similarly to paths, all cycles considered in this paper are tight unless otherwise specified.  Using \Cref{prop:avg-to-min-deg}, we obtain long paths in an $r$-digraph with many copies of $\Cr{r}$.

\begin{lem}\label{lem:cycle-linear}
Let $G$ be an $n$-vertex $r$-digraph. If $G$ contains $m$ copies of the $r$-vertex cycle $\Cr{r}$, then $G$ contains a path on at least $m/n_{(r-1)} + (r-1)$ vertices.
\end{lem}
\begin{proof}
    Let $\CC$ be the set of copies of $\Cr{r}$ in $G$.  For each $C\in\CC$, let $\psi(C)$ be the set of $(r-1)$-intervals in the cyclic ordering of $V(C)$.  Since each vertex in $C$ starts an $(r-1)$-interval, we have that $|\psi(C)| = r$.  Let $H$ be the auxiliary $r$-graph whose vertices are the $(r-1)$-tuples of distinct vertices in $V(G)$ with edge set $\{\psi(C)\st C\in\CC\}$.  Note that $|E(H)| = m$ and $|V(H)| = n_{(r-1)}$, and so $H$ has average degree $mr/n_{(r-1)}$.  Since $H$ is an $r$-graph with average degree $mr/n_{(r-1)}$, it follows that $H$ has a subgraph $H_0$ with minimum degree at least $m/n_{(r-1)}$.  Let $t$ be the minimum degree of $H_0$.
    
    Let $P$ be a longest path in $G$ with an $(r-1)$-suffix in $V(H_0)$, and let $(u_1,\ldots, u_{r-1})$ be this suffix.  Since $H_0$ has minimum degree $t$, it follows that there exist distinct vertices $v_1,\ldots,v_t\in V(G)$ such that $C_j \in\CC$ where $C_j=\cycle{u_1\ldots u_{r-1} v_j}$ and $\psi(C_j)\in E(H_0)$.  In particular, both $(u_1,\ldots,u_{r-1})$ and $(u_2,\ldots, u_{r-1}, v_j)$ are vertices in the edge $\psi(C_j)$ in $H_0$.  It follows that $v_j\in V(P)$, or else appending $v_j$ to $P$ gives a longer path ending at $(u_2,\ldots,u_{r-1},v_j)\in V(H_0)$.  Note that $|V(P)| \ge (r-1) + t$.
\end{proof}

We pause to discuss the sharpness of \Cref{lem:cycle-linear} in the case that $m=\binom{n}{r}$.  When $m=\binom{n}{r}$, \Cref{lem:cycle-linear} implies that $G$ has a path of size at least $n/(r!)$.  This is sharp up to a factor that is polynomial in $r$.

\begin{thm}
For each positive $n$ and $r$, there is an $n$-vertex $r$-digraph $G$ such that every $r$-set contains a copy of $\Cr{r}$ and every path in $G$ has size at most $2r + 2r^4(n/(r!))$.
\end{thm}
\begin{proof}
Let $F$ be the digraph whose vertices are the copies of $\Cr{r}$ in the complete $r$-digraph on $[r]$, with an edge from $C$ to $C'$ if and only if $ee'\in E(\PSG_r)$ for some $e\in E(C)$ and some $e'\in E(C')$.  Note that $F$ is the digraph obtained from $\PSG_r$ by contracting the equivalence classes of cyclic shifts and discarding loops and parallel edges.  Since the equivalence classes are disjoint and have size $r$, it follows that $|V(F)| = |V(\PSG_r)|/r = (r-1)!$.  Let $x\in \PSG_r$.  Note that $N^+(x)$ consists of the cyclic shift $x'$ of $x$ and $r-1$ other outneighbors.  When the $r$ cyclic shifts of $x$ are contracted to form a vertex $C$ in $F$, each contributes at most $r-1$ to the outdegree of $C$.  It follows that $d^+_F(C)\le r(r-1)$.  Similarly, $d^-_F(C) \le r(r-1)$.  Let $F'$ be the underlying graph of $F$ and note that $\Delta(F')\le 2r(r-1)<2r^2$.  It follows that $F'$ has an independent set $I$ with $|I|\ge |V(F')| / (1+\Delta(F')) \ge (r-1)!/(2r^2)$.

Let $G$ be the $r$-digraph on $[n]$ as follows.  Let $t=|I|$, let $I = \{C_1,\ldots,C_t\}$ and let $\{X_1,\ldots,X_t\}$ be an equipartition of $[n]$ into $t$ parts.  For each $S\in\binom{[n]}{r}$, let $\psi(S)$ be the index $j$ of the part $X_j$ containing $\min S$.  For convenience, we extend $\psi$ in the natural way to $r$-tuples of distinct vertices.  For each $S$, we include in $G$ the $r$-cycle on $S$ whose edges pattern-match the edges in $C_{\psi(S)}$, so that $(v_1,\ldots,v_r)\in E(G)$ if and only if $\canptn(v_1,\ldots,v_r)\in E(C_j)$ where $\canptn$ is the canonical pattern function and $j=\psi(v_1,\ldots,v_r)$.

Suppose that $u_1\ldots u_s$ is a path in $G$.  We claim that $\psi$ is constant on the $r$-subintervals of $u_1\ldots u_s$.  Let $e$ and $e'$ be consecutive $r$-subintervals of $u_1,\ldots,u_s$ with $e$ before $e'$.  Since the $(r-1)$-suffix of $e$ equals the $(r-1)$-prefix of $e'$, it follows that $\canptn(e) \canptn(e')\in E(\PSG_r)$.  Also, since $e,e'\in E(G)$, it follows that $\canptn(e)\in E(C_j)$ and $\canptn(e')\in E(C_{j'})$ where $j=\psi(e)$ and $j'=\psi(e')$.  If $j\ne j'$, then $F$ has an edge from $C_j$ to $C_{j'}$, but this contradicts that $C_j$ and $C_{j'}$ are both members of the independent set $I$ in the underlying graph $F'$.  It follows that $\psi$ is constant on the $r$-subintervals of $u_1\ldots u_s$.  

Let $j$ be the common value that $\psi$ assigns to each $r$-subinterval of $u_1\ldots u_s$.  It follows that each $r$-subinterval of $u_1\ldots u_s$ contains at least one vertex in $X_j$.  Therefore $s\le r|X_j| + (r-1) \le r\ceil{n/|I|} + (r-1) < r(n/|I|) + 2r \le 2r^4\cdot n/(r!) + 2r$.  
\end{proof}

Our next theorem establishes an upper bound on the linear path threshold.

\begin{thm}\label{thm:linear-thresh}
If $k = (1-\frac{1}{r} + \frac{1}{r!})r!$, then $f(n,r,k) \ge n/r!$.  It follows that the linear path threshold is at most $(1-\frac{1}{r} + \frac{1}{r!})r!$
\end{thm}
\begin{proof}
    Let $G$ be an $n$-vertex $(r,k)$-tournament.  Since each $r$-set omits fewer than $(r-1)!$ edges, it follows that each $r$-set has a copy of $\Cr{r}$.  Applying \Cref{lem:cycle-linear} with $m=\binom{n}{r}$, we obtain a path $P$ in $G$ such that $|V(P)| \ge \frac{\binom{n}{r}}{n_{(r-1)}} + r - 1 = \frac{n-r+1}{r!}+r-1 \ge n/r!$.
\end{proof}

\subsection{Spanning Path Threshold}

Using the Lovász Local Lemma~\cite{LLM} (see~\cite{alon2016probabilistic} for a textbook presentation), it is not difficult to show that if $k\ge (1-\frac{1}{e(2r-1)})r!$, then $f(n,r,k)=n$.  In this section, we relax the hypothesis to $k>(1-\frac{1}{4(r-1)})r!$. 

\begin{prop}\label{prop:large_fraction}
Let $X$ be a random variable such that $X\le M$.  If $\alpha < M$, then $\Pr(X\ge \alpha) \ge (\E(X) - \alpha)/(M-\alpha)$.
\end{prop}
\begin{proof}
    Applying Markov's inequality to the non-negative random variable $M-X$, we have $\Pr(X\ge \alpha) = \Pr(M-X \le M - \alpha) = 1-\Pr(M-X > M-\alpha) \ge 1 - \frac{\E(M-X)}{M-\alpha} = (\E(X) - \alpha)/(M-\alpha)$.
\end{proof}

\begin{prop}\label{prop:many_large_out_degree}
    Let $G$ be an $n$-vertex $r$-digraph with $n \ge r$, and let $t=\min\{n,2r-2\}$.  If the edge density $|E(G)|/n_{(r)}$ of $G$ is greater than $1-\frac{1}{4(r-1)}$, then more than half of the $(r-1)$-tuples of distinct vertices of $G$ begin more than $\frac{1}{2}n_{(t)}/n_{(r-1)}$ paths of size $t$. 
\end{prop}
\begin{proof}
    Every $r$-tuple of distinct vertices in $G$ appears consecutively in $\binom{n-r}{t-r}(t-r+1)!$ of the $t$-tuples of distinct vertices.  Since $G$ has edge density more than $1-\frac{1}{4(r-1)}$, it follows that $G$ has fewer than $\frac{1}{4(r-1)}n_{(r)}$ non-edges, and so fewer than $\frac{1}{4(r-1)}n_{(r)}\binom{n-r}{t-r}(t-r+1)!$ of the $t$-tuples of distinct vertices in $G$ have a non-edge as a substring.  Since $t\le 2r-2$, this is at most $\frac{1}{4}n_{(t)}$. So $G$ has more than $\frac{3}{4}n_{(t)}$ paths of size $t$.

    Let $\sigma$ be an $(r-1)$-tuple of distinct vertices in $G$ chosen uniformly at random, let $X$ be the number of paths in $G$ of size $t$ that begin with $\sigma$, and let $M=(n-(r-1))_{(t-(r-1))} = n_{(t)}/n_{(r-1)}$. Note that $X\le M$ for each $\sigma$, and directly applying the definition of expectation gives $\E(X) = \frac{1}{n_{(r-1)}} \sum_{\sigma} X(\sigma) > \frac{1}{n_{(r-1)}}\cdot \frac{3}{4}n_{(t)} = \frac{3}{4} M$. Let $\varepsilon = \E(X)-\frac{3}{4}M$. By \Cref{prop:large_fraction}, we have $\Pr(X > M/2) \ge \Pr(X \ge M/2+\varepsilon) \ge \frac{\E(X) - (M/2+\varepsilon)}{M-(M/2+\varepsilon)} = \frac{M/4}{M/2 -\vep}> 1/2$.
\end{proof}

\begin{thm}\label{thm:hamthrsh}
Let $n\ge r\ge 2$.  If $k> (1-\frac{1}{4(r-1)})r!$ then  $f(n,r,k) = n$.
\end{thm}
\begin{proof}
    Let $G$ be an $n$-vertex $(r,k)$-tournament with $k>(1-\frac{1}{4(r-1)})r!$.  We show that $G$ has a spanning path.  A path $P$ is \emph{flexible} if more than half of the $(r-1)$-tuples of distinct vertices in $V(G)-V(P)$ form a path when appended to $P$.  We show that $G$ has a flexible path of size $1$.  If $u$ is an $r$-tuple of distinct vertices chosen uniformly at random and $u=(u_1,\ldots,u_r)$, then $k/r! = \Pr(u\in E(G)) = \sum_{v\in V(G)} \Pr(u\in E(G) | u_1=v)\cdot \Pr(u_1 = v)$.  It follows that some vertex $v\in V(G)$ begins an edge with at least a $k/r!$ fraction of the $(r-1)$-tuples of distinct vertices in $G-v$.  Since $k/r! > 1/2$, it follows that such a vertex forms a flexible path of size $1$.

    Let $P$ be a maximal flexible path.  Let $G'=G-V(P)$, let $n'=|V(G')|$, and note that $n'\ge r-1$ since flexibility requires that at least one $(r-1)$-tuple of distinct vertices in $G'$ extends $P$.  If $n'=r-1$, then $P$ extends to a spanning path in $G$.  So we may assume $n'\ge r$.  Let $t=\min(n',2r-2)$.  Since $G'$ is an $(r,k)$-tournament with $k>(1-\frac{1}{4(r-1)})r!$, \Cref{prop:many_large_out_degree} implies that the set $A$ of $(r-1)$-tuples of distinct vertices in $G'$ which begin more than $\frac{1}{2}n'_{(t)}/n'_{(r-1)}$ paths of size $t$ has size more than $\frac{1}{2}n'_{(r-1)}$.
    
    Suppose that $n'\ge 2r-2$, and note that $t=2r-2$ in this case.  Since $P$ is flexible, the set $B$ of $(r-1)$-tuples of distinct vertices in $G'$ which form paths when appended to $P$ has size more than $\frac{1}{2}n'_{(r-1)}$.  By the pigeonhole principle, there exists $(x_1,\ldots,x_{r-1})\in A\cap B$.  Let $Q$ be the path in $G$ obtained by appending $x_1\ldots x_{r-1}$ to $P$, let $G'' = G-V(Q)$, and let $n'' = |V(G'')|$.  Since $(x_1,\ldots,x_{r-1})\in A$, there are more than $\frac{1}{2}n'_{(t)}/n'_{(r-1)}$ paths in $G'$ of size $t$ that have $x_1\ldots x_{r-1}$ as a prefix.  Since $n'_{(t)}/n'_{(r-1)} = (n'-(r-1))_{(t-(r-1))} = (n'-(r-1))_{(r-1)} = n''_{(r-1)}$, it follows that $Q$ is a larger flexible path, contradicting the maximality of $P$. 

     Therefore $r\le n' < 2r-2$.  In this case $t=n'$ and so \Cref{prop:many_large_out_degree} gives us that over half the $(r-1)$-tuples of distinct vertices of $G'$ extend to spanning paths of $G'$.  Since $P$ is flexible, at least one of these $(r-1)$-tuples extends $P$. So $P$ extends to a spanning path of $G$.
\end{proof}

\section{Small Uniformity}\label{sec:small-uniform}

In this section, we present arguments that apply to $(r,k)$-tournaments when $r$ is small.  Most of our work here is restricted to the case $r=3$ with a brief discussion of the known bounds on thresholds when $r\in\{4,5\}$.  We begin with the case $r=3$.  For readability, we write $uvw$ instead of $(u,v,w)$ to refer to an edge in a $3$-digraph.

\subsection{Paths in $3$-uniform tournaments}\label{ssec:3-uniform-tournaments}

When $r=3$, \Cref{prop:trivthrsh} implies $f(n,3,2) \le 3$, \Cref{thm:hamthrsh} gives the trivial result that $f(n,3,6) = n$, and \Cref{thm:linear-thresh} implies that $f(n,3,5)=\Theta(n)$.  In this section, we obtain improved results when $r=3$ and $k\in\{3,4,5\}$.  

\subsubsection{Paths in $(3,5)$-tournaments}

We show that every $(3,5)$-tournament has a spanning path.

\begin{prop}\label{prop:35-spanning}
$f(n,3,5)=n$.
\end{prop}
\begin{proof}
Suppose for a contradiction that the identity does not hold, and let $n$ be the smallest integer such that $f(n,3,5)<n$. Let $G$ be an $n$-vertex $(3,5)$-tournament without a spanning path.  Let $u$ be a vertex in $G$, and let $P$ be a spanning path of $G-u$, where $P=v_1v_2\ldots v_{n-1}$. 

Suppose there exists an integer $i$ such that $uv_iv_{i+1} \in E(G)$, and let $i$ be the least such integer.   Let $Q=v_1\ldots v_{i-1}u{v_iv_{i+1}}\ldots v_{n-1}$. We claim that $Q$ is a spanning path in $G$. Let $\pi$ be a $3$-interval in $Q$. If $u$ is not an entry in $\pi$, then $\pi\in E(P)\subseteq E(G)$. Otherwise consider the position of $u$ in $\pi$. If $u$ is in the first position, then $\pi=uv_iv_{i+1}\in E(G)$ by the choice of $i$. If $u$ is in the second position, then $\pi=v_{i-1}uv_{i}$. By minimality of $i$, we have that $uv_{i-1}v_{i}\not\in E(G)$.  Since $G$ is a $(3,5)$-tournament, it follows that all other permutations of $\{u,v_{i-1},v_i\}$ are edges in $G$, including $\pi$.  The case where $u$ is in the third position is similar.

Otherwise, there is no such integer $i$ and we claim that $Q$ is a spanning path, where $Q=v_1\ldots v_{n-1}u$.  Indeed, if $\pi$ is a $3$-interval containing $u$, then $\pi =v_{n-2}v_{n-1}u$.  Since $uv_{n-2}v_{n-1}\not\in E(G)$ and $G$ is a $(3,5)$-tournament, it follows that $\pi \in E(G)$.
\end{proof}

The argument in \Cref{prop:35-spanning} easily extends to show that $f(n,r,r!-1) = n$ for $r\ge 2$, but \Cref{thm:hamthrsh} already implies this when $r\ge 4$, so we present the argument only in the case $r=3$.

\subsubsection{Paths in $(3,4)$-tournaments}

In this section, we show that every $n$-vertex $(3,4)$-tournament has a path of size $\Omega(n^{1/5})$.  However, significantly stronger results may be possible.  

\begin{conj}\label{conj:34-spanning}
    Every $(3,4)$-tournament has a spanning path.  That is, $f(n,3,4)=n$.
\end{conj}

In support of this conjecture, we note that Misha Lavrov~\cite{ML} used computer search to verify \Cref{conj:34-spanning} when $n\le 7$ and obtained some evidence to suggest that perhaps \Cref{conj:34-spanning} may be further strengthened to assert that every $n$-vertex $(3,4)$-tournament has at least $2^{n-1}$ spanning paths.  This would be best possible.  Let $G$ be the $(3,4)$-tournament $[n]$ with $uvw\in E(G)$ if and only if $\max\{u,v,w\} \ne v$.  Every spanning path in $G$ consists of a spanning path of $G-n$ with $n$ added to the beginning or the end, and so $G$ has twice the number of spanning paths as $G-n$.  It follows by induction that $G$ has $2^{n-1}$ spanning paths.

To prove our lower bound on $f(n,3,4)$, we first use \Cref{lem:acyclic} to reduce to the case that $G$ has no triangles, and then we show that in such tournaments the longest paths are pairwise intersecting.  Note that if $G$ is a $(3,4)$-tournament with no triangles, then for each triple $\{u,v,w\}$ it must be that $G$ contains exactly two edges in $\{uvw,vwu,wuv\}$ and two edges in $\{wvu,vuw,uwv\}$.

\begin{lem}\label{lem:noDisjointMaxPaths}
If $G$ is an $n$-vertex $(3,4)$-tournament with no triangles, then the family of maximum paths in $G$ is pairwise intersecting.
\end{lem}
\begin{proof}
Suppose for a contradiction that $A$ and $B$ are disjoint maximum paths in $G$, with $A=a_1\ldots a_t$ and $B=b_1\ldots b_t$.  An \emph{out edge} of $A$ is an edge of the form $a_ja_{j+1}b_j$ or $a_ja_{j+1}b_{j+1}$, and an \emph{out edge} of $B$ is an edge of the form $b_jb_{j+1}a_j$ or $b_jb_{j+1}a_{j+1}$.

Suppose $G$ has out edges.  Let $a_ja_{j+1}b_{j'}$ or $b_jb_{j+1}a_{j'}$ be an out edge chosen first to maximize $j$, and next to maximize $j'\in\{j,j+1\}$.  Without loss of generality, we may assume the chosen out edge is $a_ja_{j+1}b_{j'}$.  Let $Q = a_1\ldots a_{j+1} b_{j'} \ldots b_t$.  We claim that $Q$ is a path.  Clearly, $a_ja_{j+1}b_{j'}\in E(G)$.  If $j'<t$, then $a_{j+1}b_{j'}b_{j'+1}\in E(G)$ since $b_{j'}b_{j'+1}a_{j+1} \not\in E(G)$ by our extremal choice of an out edge, and since $G$ is a $(3,4)$-tournament with no triangles, this forces $b_{j'+1}a_{j+1}b_{j'}\in E(G)$ and $a_{j+1}b_{j'}b_{j'+1}\in E(G)$.  All other consecutive triples in $Q$ are edges in $A$ or $B$.  Since $Q$ is a path in $G$ with at least $t+1$ vertices, we obtain a contradiction.

It follows that $G$ has no out edges.  Let $Q=a_tb_ta_{t-1}b_{t-1}\ldots a_2b_2a_1b_1$.  We claim that $Q$ is a path in $G$.  Indeed, if $a_jb_ja_{j-1}\not\in E(G)$, then since $G$ is a $(3,4)$-tournament with no triangles, this would force $b_ja_{j-1}a_j\in E(G)$ and $a_{j-1}a_jb_j\in E(G)$, but the latter is an out edge of $A$.  A similar argument shows that $b_ja_{j-1}b_{j-1}\in E(G)$.  Since $Q$ has more than $t$ vertices, we again obtain a contradiction.
\end{proof}

\begin{lem}\label{lem:longpath}
Let $s\ge 1$ and let $G$ be an $s^2$-vertex $(3,4)$-tournament with no triangles.  We have that $P^{(3)}_s\subseteq G$. 
\end{lem}
\begin{proof}
By induction on $s$.  The claim is clear when $s=1$.  Suppose that $s\ge 2$.  Since $|V(G)| = s^2 \ge (s-1)^2$, applying the inductive hypothesis to an arbitrary $(s-1)^2$-vertex subgraph of $G$ gives an $(s-1)$-vertex path $Q$.  Let $G'=G-V(Q)$, and note that $|V(G')|=s^2-(s-1) \ge (s-1)^2$.  Invoking the inductive hypothesis again, we have that $G'$ also contains an $(s-1)$-vertex path $Q'$.  Since $Q$ and $Q'$ are disjoint, these cannot both be maximum paths by \Cref{lem:noDisjointMaxPaths}.  Hence $G$ contains a path on $s$ vertices.
\end{proof}

We are now able to give our lower bound on $f(n,3,4)$.

\begin{thm}\label{thm:34-lower}
$f(n,3,4)\ge (1-o(1))(n/3)^{1/5}$.
\end{thm}
\begin{proof}
Let $G$ be a $(3,4)$-tournament, and let $s$ be the maximum integer such that $P^{(3)}_s\subseteq G$.  By \Cref{lem:acyclic}, $G$ contains an induced acyclic subgraph $H$ on at least $(1-o(1))\sqrt{n/(3s)}$ vertices.  By \Cref{lem:longpath}, $H$ contains a path on $(1-o(1))(n/(3s))^{1/4}$ vertices.  Hence $f(n,3,4)\ge \max\{s,(1-o(1))(n/(3s))^{1/4}\} \ge (1-o(1))(n/3)^{1/5}$.
\end{proof}

\subsubsection{Paths in $(3,3)$-tournaments}

Let $X$ be a set of $r$-tuples.  The \emph{shift graph on $X$} is the digraph with an edge from $u$ to $v$ if and only if the $(r-1)$-suffix of $u$ equals the $(r-1)$-prefix of $v$.

\begin{thm}\label{thm:(33)-upper}
$f(n,3,3)\le 2\lg n + 6$.
\end{thm}
\begin{proof}
Suppose that $n=2^t$ for some integer $t$.  We construct a $(3,3)$-tournament $G$ on $\FF_2^t$ as follows.  Given a triple $e\in\binom{V(G)}{3}$, let $\alpha(e)$ be the minimum index $\ell$ such that the vertices in $e$ are not all equal in the $\ell$ coordinate. Let $\beta(e)$ be the minimum index $\ell$ such that the pair in $e$ that agrees at coordinate $\alpha(e)$ disagrees at coordinate $\ell$.  Note that $\beta(e)>\alpha(e)$.  We extend $\alpha$ and $
\beta$ in the natural way to ordered triples.

Given an ordered triple $a$ with $a=uvw$, $i=\alpha(a)$, and $j=\beta(a)$, we include $a$ in $E(G)$ if the submatrix of the $(t\times 3)$-matrix with columns $u,v,w$ induced by rows $i$ and $j$ matches one of the following \emph{allowed patterns}.
\newcommand{\pattwobythree}[6]%
{%
    \left[ 
    \begin{array}{ccc}
        #1 & #2 & #3 \\
        #4 & #5 & #6
    \end{array} 
    \right]
}
\begin{align*}
P_1 &: \pattwobythree{0}{1}{0}{*}{*}{*} &
P_2 &: \pattwobythree{0}{0}{1}{0}{1}{*} &
P_3 &: \pattwobythree{0}{1}{1}{*}{1}{0} &
P_4 &: \pattwobythree{1}{1}{0}{*}{*}{*}
\end{align*}
In the above patterns, the asterisks match either character.  We claim that $G$ is a $(3,3)$-tournament.  Let $\{u,v,w\}$ be a triple of vertices, let $e=\{u,v,w\}$, let $i=\alpha(e)$, and let $j=\beta(e)$.  Without loss of generality, we may suppose that $u_i=v_i \ne w_i$, $u_j=0$ and $v_j=1$.  If $u_i=v_i=0$, then $uwv$ and $vwu$ match $P_1$, and $uvw$ matches $P_2$.  If $u_i=v_i=1$, then $uvw$ and $vuw$ match $P_4$, and $wvu$ matches $P_3$.  So $G$ is a $(3,3)$-tournament.

Suppose that $Q$ is path, with $Q=x_1\ldots x_m$.  Let $a_s=x_sx_{s+1}x_{s+2}$ for $1\le s\le m-2$.  We claim that if $a_s$ matches the \emph{trap pattern} $\pattwobythree{1}{1}{0}{1}{0}{*}$ with $i=\alpha(a_s)$ and $j=\beta(a_s)$, then $a_{s+1}$ also matches the trap pattern with $\alpha(a_{s+1})<\beta(a_{s+1})=i=\alpha(a_s)$.  First, note that $i$ is the minimum index such that $x_{s+1}$ and $x_{s+2}$ differ, implying $\alpha(a_{s+1}) \le i$.  For equality to hold, $a_{s+1}$ must match a pattern whose top row begins with a $1$ followed by a $0$, and this is incompatible with all allowed patterns.  Hence $\alpha(a_{s+1}) < i$, meaning that the first disagreement among vertices in $a_{s+1}$ occurs at a coordinate $\alpha(a_{s+1})$ where $x_{s+3}$ disagrees with the common value in $x_{s+1}$ and $x_{s+2}$.  Hence $\beta(a_{s+1})$ is the least coordinate where $x_{s+1}$ and $x_{s+2}$ disagree, giving $\beta(a_{s+1}) = i$.  It follows that $a_{s+1}$ matches $\pattwobythree{0}{0}{1}{1}{0}{*}$ or $\pattwobythree {1}{1}{0}{1}{0}{*}$.  Since the first of these is incompatible with the allowed patterns, the claim follows.  

Next, we claim that if $\alpha(a_s)\ne \alpha(a_{s+1})$, then either $a_s$ matches $P_2$ or $a_{s+1}$ matches the trap pattern.  Suppose that $\alpha(a_s)\ne \alpha(a_{s+1})$ and $a_s$ matches a pattern in $\{P_1,P_3,P_4\}$.  We show that $a_{s+1}$ matches the trap pattern.  Let $i=\alpha(a_s)$, let $j=\beta(a_s)$, and let $j'$ be the index of the first coordinate where $x_{s+1}$ and $x_{s+2}$ differ.  Note that $i\le j'$.  If $i=j'$, then $a_s$ does not match $P_3$ and so $a_s$ matches a pattern in $\{P_1,P_4\}$.  In this case, we have $x_{s+1}(j')=x_{s+1}(i) = 1$ and $x_{s+2}(j') = x_{s+2}(i)=0$.  Otherwise, $i<j'$, implying that that $x_{s+1}(i)=x_{s+2}(i)$.  Therefore $a_s$ matches $P_3$ and $j=j'$, again giving $x_{s+1}(j') = x_{s+1}(j) = 1$ and $x_{s+2}(j') = x_{s+2}(j) = 0$.  It follows that $\alpha(a_{s+1}) < j'$, since $\alpha(a_{s+1}) = j'$ would require $a_{s+1}$ to match a pattern whose first row begins with a 1 and then a 0, which is inconsistent with the allowable patterns.  Since $\alpha(a_{s+1}) < j'$, it follows that $x_{s+1}$ and $x_{s+2}$ agree at coordinate $\alpha(a_{s+1})$ and therefore $\beta(a_{s+1}) = j'$.  Hence $a_{s+1}$ matches $\pattwobythree{0}{0}{1}{1}{0}{*}$ or $\pattwobythree{1}{1}{0}{1}{0}{*}$.  The former is inconsistent with the allowable patterns, and so the claim follows.

A \emph{block} is a maximal interval $[a_s, \ldots,a_{s'}]$ of edges of $Q$ such that $\alpha(a_s) = \cdots = \alpha(a_{s'})$.  Partition $E(Q)$ into blocks $B_1,\ldots,B_r$.  Note that within each $B_j$, the edges match patterns that follow a walk in the shift graph on the top row of the patterns.  Note that this shift graph is acyclic, and the maximum path has order $3$.  It follows that $|B_j| \le 3$ for each $j$.  We claim there is at most one index $j$ such that the $\alpha$ component of edges in $B_j$ is less than the $\alpha$ component of edges in $B_{j+1}$.  Let $j$ be the least index such that adjacent blocks $B_j$ and $B_{j+1}$ satisfy $\alpha(a_s) < \alpha(a_{s+1})$, where $a_s$ is the last edge in $B_j$ and $a_{s+1}$ is the first edge in $B_{j+1}$.  Since $\alpha(a_s)\ne \alpha(a_{s+1})$, either $a_s$ matches $P_2$ or $a_{s+1}$ matches the trap pattern.  If $a_s$ matches $P_2$, then since $x_{s+1}$ and $x_{s+2}$ differ in $\alpha(a_s)$, it follows that $\alpha(a_{s+1}) \le \alpha(a_s)$, contradicting $\alpha(a_s) < \alpha(a_{s+1})$.  Hence $a_{s+1}$ matches the trap pattern and $\alpha(a_{s+1}) > \cdots > \alpha(a_{m-2})$.  It follows that $r\le 2t$.

Next, we claim that at most one block has size more than $1$.  Let $i$ be the least integer such that $|B_i| > 1$.  Note that if the last edge in $B_i$ matches $P_2$, then $|B_i| = 1$ since in the successor graph, the top row of $P_2$ has indegree zero.  Hence $|B_i| > 1$ implies that the last edge $a_s$ in $B_i$ does not match $P_2$ and so all edges in $Q$ after $a_s$ match the trap pattern, giving $|B_j| = 1$ for $j>i$.  It follows that $|E(Q)|\le (2t-1) + 3 = 2t+2$ and $m=|V(Q)| = |E(Q)| + 2 \le 2t + 4 = 2\lg n + 4$.  

If $n$ is not a power of two, then we apply the argument to the integer $n'$, where $n'$ is the least power of $2$ greater than $n$.  Since $n'\le 2n$, it follows that $f(n,3,3)\le f(n',3,3) \le 2\lg n' + 4 \le 2\lg n + 6$.
\end{proof}

\begin{thm}\label{thm:(33)-lower}
$f(n,3,3) \ge \Omega\left(\left(\frac{\ln n}{\ln \ln n}\right)^{1/2}\right)$.
\end{thm}
\begin{proof}
    Let $G$ be a $(3,3)$-tournament on $[n]$, and let $s$ be the size of a maximum path in $G$.  By \Cref{lem:acyclic}, there is a subtournament $H$ with $|V(H)| \ge (1-o(1))\sqrt{n/3s}$ such that every walk in $H$ has size at most $s$.  Let $n'=|V(H)|$; we may assume that $V(H)=\{1,\ldots,n'\}$.
    
    Let $F$ be a copy of the complete graph on $\{1,\ldots,n'\}$; we use $H$ to color the edges of $F$ as follows.  For each pair of vertices $u,v$ in $H$ with $u<v$, let $\ell(uv) = (a,b)$ where $a$ is the maximum length of a walk in $H$ ending $uv$ and $b$ is the maximum length of a walk in $H$ ending $vu$.  We claim that there is no monochromatic triangle in $F$ under the coloring $\ell$.  Suppose for a contradiction that $uvw$ is a monochromatic triangle with $u<v<w$.  Let $(a,b)$ be the common color on $uvw$ in $F$.  Note that $uvw\not\in E(H)$, since a walk of length $a$ ending at $uv$ would extend to a walk of length $a+1$ ending at $vw$, contradicting that $\ell(vw)$ has first component $a$.  Similarly, $wvu\not\in E(H)$ as both $\ell(vw)$ and $\ell(uv)$ have equal second component $b$.  Since $H$ is a $(3,3)$-tournament, it follows that $E(H)$ contains $\{vwu,wuv\}$ or $E(H)$ contains $\{uwv,vuw\}$.  
    
    Suppose that $\{vwu,wuv\}\subseteq E(H)$.  Since a walk ending $vw$ of length $a$ extends to a walk ending $wu$ of length $a+1$, it follows that $b\ge a+1$.  Also, since $wuv\in E(H)$, a walk ending $wu$ of length $b$ extends to a walk ending $uv$ of length $b+1$, and so $a\ge b+1$.  This is a contradiction, and the case that $\{uwv,vuw\}\subseteq E(H)$ leads to a similar contradiction.

    Recall that the $p$-color Ramsey number for triangles $R(3;p)$ satisfies $R(3;p)\le 3p!$ (see~\cite{Ramsey-survey} for an introductory survey).  Since walks in $H$ have size at most $s$, the lengths range from $0$ to $s-(r-1)$, and so $\ell$ gives a $s^2$-edge-coloring of a copy of $K_{n'}$ with no monochromatic triangles.  It follows that $(1-o(1))\sqrt{n/3s}\le n'<R(3;s^2) \le 3(s^2)! \le 3(s^2)^{s^2}=3s^{2s^2}$.  Taking the natural log of both sides gives $(1/2)(\ln(n/(3s))) - o(1) \le \ln(3) + 2s^2\ln(s)$ or $(1/2)(\ln n) - o(1) \le (3/2)\ln(3) + (2s^2 + (1/2))\ln(s)$.  If $s\ge \ln n$, then $G$ has a path on $\ln n$ vertices and the bound follows.  Otherwise $s<\ln n$ and it follows that $(1/2)\ln n - o(1) \le (3/2)\ln(3) + (2s^2 + (1/2))\ln\ln n$ and so $\frac{\ln n}{2\ln\ln n} -o(1) \le 2s^2 +  (1/2) \le 3s^2$.  Hence $s\ge [\frac{\ln n}{6\ln\ln n} - o(1)]^{1/2}$.  
\end{proof}

\subsection{Paths in $r$-uniform tournaments for $r\in\{4,5\}$}

In this section, we summarize known results for $r\in\{4,5\}$ in \Cref{fig:thresholds}.  Most of these are obtained by applying general theorems in \Cref{sec:2-Thresholds-general-r}.  To obtain the threshold for growing paths when $r\in\{4,5\}$, we find $\trans{}{\PSG_4}$ and $\trans{}{\PSG_5}$.  

A list of integers $(u_1,\ldots,u_r)$ has a \emph{bump} at position $i$ if $u_i > u_{i-1}$ and $u_i > u_{i+1}$.

\newcommand{\drawPSGFour}%
{
    \begin{tikzpicture}[]
        \begin{scope}[every node/.style={circle,draw,black,inner sep=1pt,minimum size=4ex},line width=0.7pt,anchor=north, font=\small]

            \foreach \n/\lab in {0/3142,1/1423,2/4231,3/2314}
            {{
                \pgfmathsetmacro{\ang}{(3-\n)*90}
                \node at (\ang:2.15cm) (U\lab) {$\lab$} ;
            }}

            \foreach \n/\lab in {0/3241, 1/2413, 2/4132, 3/1324}
            {{
                \pgfmathsetmacro{\ang}{\n*90}
                \node at (\ang:0.85cm) (U\lab) {$\lab$} ;
            }}

            \begin{scope}[dotted,decoration={markings,mark=at position 0.5 with {\arrow{Stealth}}},every path/.style={bend left=30}]
                \foreach \u/\v in {%
                        U1423/U4231, U4231/U2314, U2314/U3142, U3142/U1423%
                    }
                {{
                    \draw[postaction={decorate}] (\u) to (\v) ;
                }}
            \end{scope}

            \begin{scope}[dotted,decoration={markings,mark=at position 0.75 with {\arrow{Stealth}}},every path/.style={bend right=30}]
                \foreach \u/\v in {%
                        U3241/U2413, U2413/U4132, U4132/U1324, U1324/U3241%
                    }
                {{
                    \draw[postaction={decorate}] (\u) to (\v) ;
                }}
            \end{scope}

            \begin{scope}[decoration={markings,mark=at position 0.75 with {\arrow{Stealth}}},every path/.style={bend left=15}]
                \foreach \u/\v in {
                    U2413/U4231, U4231/U2413, U3241/U2314, U2314/U3241, U1324/U3142, U3142/U1324, U1423/U4132, U4132/U1423%
                    }
                {{
                    \draw[postaction={decorate}] (\u) to (\v) ;
                }}
            \end{scope}
        \end{scope}
    \end{tikzpicture}
}
\begin{figure}
    \begin{center}
    \drawPSGFour
    \end{center}
    \caption{The two shift cycles above (dotted edges) in $\PSG_4$ are replaced with four $2$-cycles (solid edges) in our construction of a family of disjoint cycles. }\label{fig:PSG4}
\end{figure}

\begin{prop}\label{prop:PSG4-trans}
    We have $\trans{}{\PSG_4} = 10$ and $\acyc(\PSG_4) = 4! - \trans{}{\PSG_4} = 14.$
\end{prop}
\begin{proof}
Note that by \Cref{prop:max-elt-travel}, for $r\ge 3$, there is no edge $uv$ in $\PSG_r$ with $u=(u_1,\ldots,u_r)$, $v=(v_1,\ldots,v_r)$, $u_r = r$, and $v_1 = r$.

We construct a cycle transversal $S$ in $\PSG_4$ of size $10$.  Let $S$ consist of the vertices $1234$, $4321$, and the permutations $(a_1,a_2,a_3,a_4)$ of $[4]$ such that $\max\{a_1,a_2,a_3\}=a_2$.  Note that $|S| = 2 + (1/3)\cdot 4! = 10$.  Let $C$ be a cycle in $\PSG_4$.  Suppose for a contradiction that $V(C)$ and $S$ are disjoint.  Since $V(C)$ and $S$ are disjoint, no vertex in $C$ has a bump in the second position.  It follows that no vertex $u\in V(C)$ has a bump in the third position either, since then the successor of $u$ in $C$ would have a bump in the second position.  Therefore for each $u\in V(C)$ with $u=(u_1,\ldots,u_4)$, we have that $u_1 = 4$ or $u_4 = 4$.  Since $C$ has no edge $uv$ with $u=(u_1,\ldots,u_4)$ and $v=(v_1,\ldots,v_4)$ where $u_4 = 4$ but $v_1 = 4$, it must be that either every vertex in $C$ has $4$ in the first position, forcing $C$ to be the loop $\cycle{4321}$, or every vertex in $C$ has $4$ in the fourth position, forcing $C$ to be the loop $\cycle{1234}$.  It follows that $\trans{}{\PSG_4} \le 10$.

To show $\trans{}{\PSG_4} \ge 10$, we construct a disjoint family $\CC$ of cycles in $\PSG_4$ with $|\CC| = 10$.  Note that the shift cycles partition $V(\PSG_4)$ into $6$ subgraphs containing $C_4$.  To produce four more cycles, we modify the partition as follows.  As in \Cref{lem:split-cycle}, we split $\cycle{1234, 2341, 3412, 4123}$ into the loop $\cycle{1234}$ and the $3$-cycle $\cycle{2341, 3412, 4123}$.  Similarly, we split $\cycle{4321, 3214, 2143, 1432}$ into the loop $\cycle{4321}$ and the $3$-cycle $\cycle{3214, 2143, 1432}$.  Finally, we use the vertices in the shift cycles $\cycle{1324, 3241, 2413, 4132}$ and $\cycle{1423, 4231, 2314, 3142}$ to produce four $2$-cycles: $\cycle{1324, 3142}$, $\cycle{3241, 2314}$, $\cycle{2413, 4231}$, $\cycle{4132, 1423}$; see \Cref{fig:PSG4}.  We leave in place the other two shift cycles, giving 
\begin{align*}
    \CC = \{&\cycle{1234}, \cycle{4321}, \cycle{1324, 3142}, \cycle{3241, 2314}, \cycle{2413, 4231}, \cycle{4132, 1423}, \\
        &\cycle{2341, 3412, 4123}, \cycle{3214, 2143, 1432}, \cycle{1243, 2431, 4312, 3124}, \cycle{1342, 3421, 4213, 2134}\}.
\end{align*}
\end{proof}

\begin{figure}
    \begin{center}
    \begin{tabular}{l|c|l}
    Condition  &    Threshold for $k$  & Reference\\
    \hline
    $f(n,4,k)\ge 6$ & $13$ & \Cref{thm:PSG-t-const} with $(r,t)=(4,3)$ and \Cref{thm:const-path-thresh} \\
    $f(n,4,k) \ge \omega(1)$ & 15 & \Cref{prop:PSG4-trans} and \Cref{thm:growing-path-tourn-threshold} \\
    $f(n,4,k) \ge \Omega(n)$ & $[15,19]$ & \Cref{thm:linear-thresh} \\
    $f(n,4,k) = n$ & $[15,23]$ & \Cref{thm:hamthrsh} \\
    && \\
    $f(n,5,k)\ge 6$ & $49$ & \Cref{thm:PSG-t-const} with $(r,t)=(5,2)$ and \Cref{thm:const-path-thresh} \\
    $f(n,5,k)\ge 8$ & $73$ & \Cref{thm:PSG-t-const} with $(r,t)=(5,4)$ and \Cref{thm:const-path-thresh} \\
    $f(n,5,k) \ge \omega(1)$ & 85 & \Cref{prop:PSG5-trans} and \Cref{thm:growing-path-tourn-threshold} \\
    $f(n,5,k) \ge \Omega(n)$ & $[85,97]$ & \Cref{thm:linear-thresh} \\
    $f(n,5,k) = n$ & $[85,113]$ & \Cref{thm:hamthrsh}    
    \end{tabular}
    \end{center}
    \caption{\label{fig:thresholds}Bounds on thresholds in $4$-uniform and $5$-uniform tournaments}
\end{figure}

\begin{prop}\label{prop:PSG5-trans}
    We have $\trans{}{\PSG_5} = 36$ and $\acyc(\PSG_5) = 5! - \trans{}{\PSG_5} = 84$.
\end{prop}
\begin{proof}
For the upper bound, let $S$ be the set of vertices in $\PSG_5$ consisting of the identity $12345$, its reverse $54321$, and vertices $(u_1,\ldots,u_5)$ with a bump at position $2$.  Let $S_0$ be the vertices $(u_1,\ldots,u_5)$ that have bumps in positions 2 and 4 and have $u_3=1$.  Note that $|S|=2+(5!/3)=42$ and $|S_0| = \binom{4}{2}=6$ since a vertex $(u_1,\ldots,u_5)\in S_0$ is determined by the choice of the unordered pair $\{u_1,u_2\}$ from $\{2,3,4,5\}$.  Let $S' = S-S_0$ and note that $|S'| = |S| - |S_0| = 36$.  We claim that $S'$ is a cycle transversal.  Let $C$ be a cycle in $\PSG_5$.  Suppose that no vertex in $C$ has a bump and let $u\in V(C)$ with $u=(u_1,\ldots,u_5)$.  If $u=54321$, then $C$ and $S'$ intersect.  Otherwise $u_i < u_{i+1}$ for some $i$ with $1\le i\le 4$ and it follows that $(u_i,\ldots,u_5)$ is increasing since $u$ has no bump.  Moreover, since the successor $v$ of $u$ in $C$ also has no bump, it follows that $(v_{i-1},\ldots,v_5)$ is increasing; eventually, we see that $12345 \in V(C)$ and so $C$ and $S'$ intersect.  

Hence we may assume that some vertex in $C$ has a bump, and by taking successors it follows that $C$ contains a vertex $u$ with a bump at position $2$.  If $u\not\in S_0$, then $C$ and $S'$ intersect.  So we may assume that $u\in S_0$ and so $u=(u_1,\ldots,u_5)$ where $u$ has a bump at positions $2$ and $4$ with $u_3=1$.  Note that $u_3<u_5$.  Advancing twice along $C$ gives a vertex $v$ with $v=(v_1,\ldots,v_5)$ such that $v_1 < v_3$ and $v$ has a bump in position $2$.  It follows that $v\in S$.  Also, since $v_3>v_1$ we cannot have $v_3=1$ and so $v\not\in S_0$.  It follows that $v\in S'$ and so $C$ and $S'$ intersect.  Since $C$ and $S'$ intersect in all cases, $S'$ is a cycle transversal and so $\trans{}{\PSG_5} \le |S'| = 36$.

For the lower bound, we present a disjoint family $\CC$ of cycles in $\PSG_5$ with $|\CC| = 36$ and $\bigcup_{C\in\CC} V(C) = V(\PSG_5)$.  The family was obtained via computer search.  The family $\CC$ has $t_j$ cycles of length $j$, where $(t_1,\ldots,t_5) = (2,6,8,18,2)$.  For readability, we group the cycles by their lengths.

\begin{align*}
\CC = \{&\cycle{12345},\cycle{54321}, \\
\\
&\cycle{13254,21435},\cycle{14253,31425},
\cycle{15243,51423},\cycle{34251,32415},\cycle{35241,52413},\cycle{45231,53412},\\
\\
&\cycle{12435,13245,21354},\cycle{13524,25134,51342},\cycle{14523,34125,41352},\cycle{15324,53142,41532}\\
&\cycle{32541,25314,52143},\cycle{23514,24135,42351},\cycle{24315,43152,42531},\cycle{45312,54231,53421}\\
\\
&\cycle{12354,12534,14235,31245},\cycle{12453,13425,23145,21345},
\cycle{12543,15423,54123,51243},\\
&\cycle{32451,34521,34215,32145},\cycle{35421,54312,54132,52431},
\cycle{43521,45321,54213,53241},\\ &\cycle{13452,24513,35124,41235},\cycle{23451,34512,45123,51234},
\cycle{13542,25413,43125,31254},\\ &\cycle{23541,35412,53124,41253},\cycle{14352,32514,25143,41325},
\cycle{24351,42513,35142,51324},\\ &\cycle{14532,35214,42135,21453},\cycle{24531,45213,52134,31452},
\cycle{15342,42315,24153,31524},\\ &\cycle{25341,52314,34152,41523},\cycle{15432,43215,32154,21543},
\cycle{25431,53214,42153,31542},\\
\\
&\cycle{15234,52341,23415,23154,21534},\cycle{43251,43512,45132,51432,14325}\}
\end{align*}
\end{proof}

It would be interesting to have general constructions of large disjoint families of cycles in $\PSG_r$ for general $r$.  We note that disjoint cycles in $\PSG_r$ correspond to edge-disjoint cycles in the multigraph variant of $\PSG_{r-1}$ where each shift edge has multiplicity $2$ (and other edges have multiplicity $1$); like $\PSG_{r-1}$, this multigraph variant is an Eulerian digraph and hence decomposes into cycles.  Unfortunately, some cycles in the decomposition may be long, and the size of the decomposition may be small.  We also note that each vertex in $\PSG_r$ is contained in an $(r-1)$-cycle.  However, these shorter cycles do not seem to easily lead to disjoint families like the shift cycles do.

\section{Density Results}\label{sec:density}

The \emph{falling factorial}, denoted $n_{(r)}$, is $n(n-1)\cdots (n-(r-1))$.  In a \emph{complete $r$-digraph}, every $r$-tuple of distinct vertices is an edge; equivalently, a complete $r$-digraph is an $(r,r!)$-tournament.  The complete $n$-vertex $r$-digraph has $n_{(r)}$ edges.

In this section, we show that if $G$ is an $n$-vertex $r$-digraph with $|E(G)| \ge (1-\frac{1}{r})n_{(r)}$, then $G$ contains a path whose size increases with $n$.  Conversely, for each positive $\vep$, there are infinitely many $r$-digraphs $G$ with $|E(G)| \ge (1-\frac{1}{r}-\vep)n_{(r)}$, where $n=|V(G)|$, but all paths in $G$ have size at most $c$, where $c$ is a constant depending only on $\vep$.  In this sense, $1-\frac{1}{r}$ is the \emph{density threshold} for growing paths in $r$-digraphs.

\begin{lem}\label{lem:density-lower}
Let $0 < \vep < 1$ and $r\ge 2$.  For infinitely many $n$, there exists an $n$-vertex $r$-digraph $G$ with $|E(G)| \ge (1-\frac{1}{r}-\vep)n_{(r)}$ such that every path in $G$ has at most $\frac{r^3}{\vep}$ vertices.
\end{lem}
\begin{proof}
Fix a parameter $t$, and let $n$ be a multiple of $t$.  We construct an $r$-digraph $G$ with vertex set $[n]$ as follows.  Partition $[n]$ into $t$ intervals $X_1,\ldots,X_t$ of equal size.  Given an $r$-tuple $(u_1,\ldots,u_r)$, we put $(u_1,\ldots,u_r)\in E(G)$ if and only if $u_1 \in X_p$ and $u_j \in X_q$ for some $j,p,q$ with $p<q$.  By construction, if $v_1\ldots v_s$ is a path in $G$, then we obtain a subsequence of vertices belong to parts with increasing indices of size at least $\ceil{s/(r-1)}$.  It follows that $s\le t(r-1)$.

To bound $|E(G)|$, consider choosing an $r$-set $R$ from $[n]$ at random, and then selecting a random ordering of $R$ to obtain an $r$-tuple $(u_1,\ldots,u_r)$.  Let $A$ be the event that the vertices in $R$ belong to distinct parts, and note that $\Pr(A) \ge \prod_{j=0}^{r-1}\left(1-\frac{j}{t}\right) > \left(1-\frac{r}{t}\right)^r$.  Let $B$ be the event that $(u_1,\ldots,u_r)\in E(G)$.  Conditioned on $A$, we have $\Pr(B|A) = 1-\frac{1}{r}$ since ordering $R$ yields an edge in $G$ unless the vertex belonging to the highest-indexed part is placed in the first position.  It follows that $|E(G)| = \Pr(B)\cdot n_{(r)} \ge \Pr(B|A)\Pr(A)\cdot n_{(r)} > \left(1-\frac{1}{r}\right)\left(1-\frac{r}{t}\right)^r \cdot n_{(r)}$.

We choose $t$ large enough so that $\left(1-\frac{1}{r}\right)\left(1-\frac{r}{t}\right)^r \ge 1-\frac{1}{r} - \vep$; using calculus, we have that $t\ge \frac{r^2}{\vep}$ suffices, so set $t=\ceil{r^2/\vep}$.  Since every path in $G$ has at most $t(r-1)$ vertices and $t(r-1)\le (\frac{r^2}{\vep} + 1)(r-1) \le \frac{r^3}{\vep}$, the claimed bound follows.
\end{proof}

By \Cref{lem:cycle-linear}, if $G$ has $\omega(n^{r-1})$ copies of $C^{(r)}_r$, then $G$ has a path on $\omega(1)$ vertices.  Let $G$ be an $n$-vertex $r$-digraph.  A set $S$ of $r$ vertices is \emph{sparse} if $G[S]$ has fewer than $(1-\frac{1}{r})r!$ edges, is \emph{dense} if $G[S]$ has more than $(1-\frac{1}{r})r!$ edges, and is \emph{balanced} if $G[S]$ has exactly $(1-\frac{1}{r})r!$ edges.  Since $G[S]$ contains a copy of $C^{(r)}_r$ when $S$ is dense, $n$-vertex $r$-digraphs with $\omega(n^{r-1})$ dense $r$-sets contain growing paths.

\begin{lem}\label{lem:density-upper}
  For all positive integers $r$ and $s$ with $r\ge 2$, there is a constant $n_0=n_0(r,s)$ such that if $n\ge n_0$ and $G$ is an $n$-vertex $r$-digraph and $P^{(r)}_s \not\subseteq G$, then $|E(G)| < (1-\frac{1}{r})\cdot n_{(r)}$.
\end{lem}
\begin{proof}
Let $\alpha$, $\beta$, and $\gamma$ be the fraction of $r$-sets in $V(G)$ which are sparse, balanced, and dense, respectively.  Since each dense $r$-set contains a copy of $C^{(r)}_r$, it follows from \Cref{lem:cycle-linear} that the number of dense $r$-sets $\gamma\binom{n}{r}$ satisfies $\gamma\binom{n}{r} \le (s-r+1)n_{(r-1)}$ or equivalently $\gamma \le \frac{s-r+1}{n-r+1} r!= O(1/n)$.  Hence $\gamma \to 0$ as $n\to \infty$.

Let $H$ be the $r$-graph on $V(G)$ such that $e\in E(H)$ if and only if $e$ is a balanced or dense $r$-set in $G$, let $L$ be a maximum clique in $H$, and let $\ell = |L|$.  Note that each $r$-set $S$ of vertices in $L$ is an edge in $H$ and hence a dense or balanced $r$-set in $G$.  It follows that $G[L]$ contains an $\ell$-vertex $(r,k)$-tournament $G_0$, where $k=[(1-1/r)r!]$.  Since $k>\acyc(\PSG_r)$ and $G_0$ has no path on $s$ vertices, it follows from \Cref{thm:const-path-thresh} that $\ell$ is a bounded by a function of the constants $r$ and $s$.  Since $H$ is an $n$-vertex $r$-graph with no clique of size $\ell+1$, it follows from de Caen's bound~\cite{deCaen} that $|E(H)|\le (1-1/\binom{\ell}{r-1} + o(1))\binom{n}{r}$.  Since $|E(H)| = (\beta + \gamma)\binom{n}{r}$, the number of non-edges in $H$ is $\alpha\binom{n}{r}$, and it follows that $\alpha\binom{n}{r} \ge (1/\binom{\ell}{r-1} - o(1))\binom{n}{r}$ and hence $\alpha/r! \ge (1/(r\ell_{(r-1)})) - o(1)$.

We compute 
\begin{align*}
|E(G)| &\le \alpha\binom{n}{r}\cdot \left(1-\frac{1}{r}-\frac{1}{r!}\right)r! + \beta\binom{n}{r}\cdot \left(1-\frac{1}{r}\right)r! + \gamma\binom{n}{r}\cdot r!\\
&= \left(\alpha\left(1-\frac{1}{r}-\frac{1}{r!}\right) + \beta \left(1-\frac{1}{r}\right) + \gamma\right)n_{(r)}\\
&= \left(\left(1-\gamma\right)\left(1-\frac{1}{r}\right) + \gamma - \frac{\alpha}{r!}\right)n_{(r)}\\
&= \left(1-\frac{1}{r} - \frac{\alpha}{r!} + \frac{\gamma}{r}\right)n_{(r)}\\
&\le \left(1-\frac{1}{r} - \frac{1}{r\ell_{(r-1)}} + \frac{\gamma}{r} + o(1)\right)n_{(r)}.
\end{align*}
Hence $|E(G)|<(1-(1/r))n_{(r)}$ provided that $\frac{\gamma}{r} + o(1) < \frac{1}{r\ell_{(r-1)}}$.  Recalling that $\ell$ is a constant bounded by a function of $r$ and $s$ and $\gamma\to 0$ as $n\to\infty$, the lemma follows by taking $n_0$ large enough.
\end{proof} 

Taking the contrapositive of \Cref{lem:density-upper} gives the following.
\begin{cor}\label{cor:density-upper}
For all positive integers $r$ and $s$ with $r\ge 2$, there is a constant $n_0 = n_0(r,s)$ such that if $n\ge n_0$ and $G$ is an $n$-vertex $r$-digraph with $|E(G)| \ge (1-\frac{1}{r})n_{(r)}$, then $P^{(r)}_s \subseteq G$.
\end{cor}

Recall from \Cref{cor:growing-threshold-bounds} that when $k=(1-\frac{1}{r} - \frac{\varphi(r)-1}{r!})r!$, we have $k>\acyc(\PSG_r)$ and so the $(r,k)$-tournaments have growing paths.  These tournaments have $(1-\frac{1}{r}-\frac{\varphi(r)-1}{r!})n_{(r)}$ edges, corresponding to an edge density of $1-\frac{1}{r} - \frac{\varphi(r)-1}{r!}$.  For $r\ge 3$, we have $\varphi(r)\ge 2$ and so at this density, growing paths are forced in tournaments but they are not quite yet forced in general $r$-digraphs.

\section{Conclusions}

Although obtaining the exact threshold on $k$ for growing paths in $(r,k)$-tournaments appears to be difficult, obtaining the exact threshold on $k$ for growing cycles in $(r,k)$-tournaments is easy.

\begin{prop}
Let $k = (1-\frac{1}{r} + \frac{1}{r!})r!$.  If $G$ is an $n$-vertex $(r,k)$-tournament, then $C^{(r)}_{n'}\subseteq G$ for some $n'$ that grows with $n$.
\end{prop}
\begin{proof}
Let $G$ be an $(r,k)$-tournament on vertex set $[n]$. Every $r$-set of vertices contains a copy of $C^{(r)}_r$.  Label each $r$-set $X$ by the canonical pattern of some edge in an $r$-cycle in $G[X]$. By Ramsey Theory, there is a subgraph $H$ of $G$ such that $|V(H)| = n'$ where $n'$ grows with $n$ and is a multiple of $r$, and every $r$-set in $V(H)$ is labeled with a common canonical pattern $\pi$. A cycle spanning $H$ can be constructed as follows. Let $A_1,\ldots,A_r$ be a partition of $V(H)$ into $r$ intervals of equal size such that $i<j$ implies $\max(A_i) < \min(A_j)$.  There is a cyclic arrangement $v_1,\ldots,v_{n'}$ of $V(H)$ such that each $r$-interval consists of vertices in distinct parts in $\{A_1,\ldots,A_r\}$, and for each $i$, we have that $\canptn(v_{i+1},\ldots,v_{i+r})$ is a cyclic shift of $\pi$.  Hence $\cycle{v_1,\ldots,v_{n'}}$ is a spanning cycle in $H$ whose size $n'$ grows with $n$.
\end{proof}

Since the $r!$ tuples on an $r$-set decompose into $(r-1)!$ copies of $\Cr{r}$, avoiding $\Cr{r}$ requires omitting at least $(r-1)!$ edges from each $r$-set.  In fact, when $k=r! - (r-1)!$, it is possible to avoid all closed walks.

\begin{prop}\label{prop:acyclic-tournament}
Let $k=(1-\frac{1}{r})r!$.  For each $n$, there exists an $n$-vertex $(r,k)$-tournament with no closed walk.
\end{prop}
\begin{proof}
    Construct an $(r,k)$-tournament $G$ on $[n]$ by omitting $(u_1,\ldots,u_r)$ from $E(G)$ if and only if $u_1 = \max\{u_1,\ldots,u_r\}$.  For each $r$-set, we omit $(r-1)!$ edges from $G$, and so $G$ is an $(r,k)$-tournament with $k=r!-(r-1)!=(1-\frac{1}{r})r!$.  Suppose $W$ is a closed walk in $G$ with $W=v_1\ldots v_\ell$.  Let $v_i = \max\{v_1,\ldots,v_\ell\}$ and note that $v_i$ begins an edge in $W$, contrary to the definition of $G$.
\end{proof}

Many interesting open problems remain; perhaps most compelling is the $(3,4)$-tournament conjecture (\Cref{conj:34-spanning}).  It would be notable progress to establish the $(3,4)$-tournament conjecture even for tournaments without closed walks.

\begin{conj}\label{conj:34-acyclic}
If $G$ is a $(3,4)$-tournament with no closed walk, then $G$ has a spanning path.
\end{conj}

Proving \Cref{conj:34-acyclic} would imply $f(n,3,4)\ge \Omega(n^{1/3})$, improving \Cref{thm:34-lower}.  It would also be very interesting to generalize some of the techniques in \Cref{ssec:3-uniform-tournaments} to larger $r$.  Our next conjecture requests a generalization of \Cref{thm:34-lower}.

\begin{conj}\label{conj:poly-thresh}
Let $k=(1-\frac{1}{r})r!$.  There is a positive $\vep$ such that $f(n,r,k) \ge \Omega(n^{\vep})$. 
\end{conj}

We note that \Cref{thm:linear-thresh} shows that when $k$ is just one larger than the value in \Cref{conj:poly-thresh}, we have that $f(n,r,k)$ is linear in $n$.

\section{AI Declaration}
During the preparation of this work the authors used AI tools such as ChatGPT, Gemini, and Claude to assist in a search for large families of disjoint cycles in $\PSG_r$ and in an attempt to find a counter-example or proof of the $(3,4)$-tournament conjecture.  The text in this article was written by the authors and the authors take full responsibility for the content of this article.

Subsequent to the peer review of this paper, the authors used ChatGPT 5.6 Sol to obtain the following improvement to \Cref{thm:hamthrsh}: if $k\ge (1-\frac{1}{er})r!$, then $f(n,r,k)=n$.  The proof is available in the appendix in the arXiv version of this paper.

\section*{Appendix}

Chat GPT 5.6 Sol found the following argument which simplifies and strengthens \Cref{thm:hamthrsh}.  The authors edited the original engine response for clarity and brevity.

\begin{lem}[Chat GPT 5.6 Sol]\label{lem:path_prob_inductive}
Let $G$ be an $n$-vertex $(r,k)$-tournament and let $q$ be the fraction of non-edges in $G$.  For $0\le m\le n$, let $p_m$ be the probability that a list of $m$ distinct vertices chosen at random forms a path.  We have $p_m \ge p_{m-1} - qp_{m-r}$ for $r<m\le n$.
\end{lem}
\begin{proof}
    Let $L$ be a list of $m$ distinct vertices in $G$, with $L =u_1\ldots u_m$, let $A$ be the event that $L$ is a path, let $B$ be the event that $u_1\ldots u_{m-1}$ is a path.  Since $A\subseteq B$, we have that $\Pr(B-A)=\Pr(B)-\Pr(A)=p_{m-1}-p_m$.  Let $C$ be the event that $u_1\ldots u_{m-r}$ is a path and let $D$ be the event that $(u_{m-r+1},\ldots,u_m)\not\in E(G)$.  Note that $B-A \subseteq C\cap D$ and $\Pr(D~|~C) = \Pr(D) = q$ since every ordering of $\{u_{m-r+1},\ldots,u_m\}$ is equally likely to occur and the fraction of non-edges in this $r$-set is $q$.  It follows that $p_{m-1} - p_m = \Pr(B-A) \le \Pr(C\cap D) = \Pr(D~|~C)\cdot \Pr(C) = qp_{m-r}$. 
\end{proof}

\begin{lem}[Chat GPT 5.6 Sol]\label{lem:path_prob_explicit}
Let $G$ be an $n$-vertex $(r,k)$-tournament and let $q$ be the fraction of non-edges in $G$.  If $0\le \lambda \le 1$ and $\lambda^{r-1}(1-\lambda) \ge q$, then for $1\le m\le n$ the probability $p_m$ that a list of $m$ distinct vertices chosen at random forms a path satisfies $p_m \ge \lambda p_{m-1}$.  
\end{lem}
\begin{proof}
    By induction on $m$.  Note that we may assume $\lambda > 0$.  If $m\le r-1$, then every list of $m$ distinct vertices in $G$ forms a path, giving $p_m = 1 \ge \lambda p_{m-1}$.  If $m=r$, then $p_m = 1-q \ge \lambda \ge \lambda p_{m-1}$, where $1-q\ge \lambda$ follows from $1-\lambda \ge \lambda^{r-1}(1-\lambda) \ge q$.

    Suppose that $m > r$.  Note that $p_{m-1} \ge \lambda p_{m-2} \ge \lambda^2 p_{m-3} \ge \cdots \ge \lambda^{r-1} p_{m-r}$.  By \Cref{lem:path_prob_inductive}, we have $p_m \ge p_{m-1} - qp_{m-r} \ge p_{m-1} - q\frac{p_{m-1}}{\lambda^{r-1}} = p_{m-1}(1-\frac{q}{\lambda^{r-1}}) \ge p_{m-1}(1-\frac{\lambda^{r-1}(1-\lambda)}{\lambda^{r-1}}) = p_{m-1}\lambda$.
\end{proof}

\begin{thm}[Chat GPT 5.6 Sol]\label{thm:chatgpt-spanning-bound}
    If $k\ge (1-\frac{1}{er})r!$, then every $(r,k)$-tournament has a spanning path.
\end{thm}
\begin{proof}
We may assume $r\ge 2$.  Let $G$ be an $n$-vertex $(r,k)$-tournament and let $q=1-k/(r!)$, so that $q$ is the fraction of non-edges in $G$.  We show that for $\lambda=1-1/r$, we   have $\lambda^{r-1}(1-\lambda) \ge q$, and it will then follow from \Cref{lem:path_prob_explicit} that the probability that a random permutation of $V(G)$ gives a spanning path is at least $\lambda^n$.  Since $\lambda > 0$, it follows that $G$ has a spanning path.

Hence it suffices to show $\lambda^{r-1}(1-\lambda)\ge q$.  We compute $\lambda^{r-1}(1-\lambda) = (1-\frac{1}{r})^{r-1}\frac{1}{r} \ge \frac{1}{er}$.  Since $q=1-k/(r!)$, it suffices to show that $\frac{1}{er} \ge 1-k/(r!)$, and this inequality is equivalent to $k\ge (1-\frac{1}{er})r!$.
\end{proof}

The hypothesis in \Cref{thm:chatgpt-spanning-bound} can be relaxed slightly to $k\ge r!(1-\frac{1}{r}(1-\frac{1}{r})^{r-1})$ by not using the inequality $(1-\frac{1}{r})^{r-1} \ge \frac{1}{e}$.

\printbibliography

\end{document}